\documentclass{amsart}
\usepackage{geometry}
\usepackage{amsfonts}
\usepackage{hyperref}
\usepackage{amsmath}
\usepackage{galois}
\DeclareMathSizes{10}{10}{7}{5}
\usepackage{mathtools}
\numberwithin{equation}{section}

\usepackage{enumitem}
\usepackage{amssymb}
\newtheorem{theorem}{Theorem}[section]
\newtheorem{lemma}{Lemma}[section]
\newtheorem{proposition}{Proposition}[section]

\newtheorem{example}{Example}[section]
\newtheorem{remark}{Remark}[section]
\newtheorem*{problem}{Problem}
\begin{document}
\title[\textbf{Obstacles to periodic orbits hidden at fixed point}]{
\textbf{Obstacles to periodic orbits hidden at fixed point of holomorphic maps}\protect\footnotemark[2]}
\author{Jianyong Qiao}
\author{Hongyu Qu \textsuperscript{*}}
\address{School of Sciences, Beijing University of Posts and Telecommunications, Beijing
100786, P. R. China. \textit{Email:} \textit{qjy@bupt.edu.cn}}
\address{School of Sciences, Beijing University of Posts and Telecommunications, Beijing
100786, P. R. China. \textit{Email:} \textit{hongyuqu@bupt.edu.cn}}
\renewcommand{\thefootnote}{\fnsymbol{footnote}}
\footnotetext[2]{The two authors are co-first authors.}
\footnotetext[1]{Corresponding author, Email: hongyuqu@bupt.edu.cn}
\maketitle
\begin{abstract}
Let $f:(\mathbb{C}^n,0)\mapsto(\mathbb{C}^n,0)$ be a germ of an $n$-dimensional holomorphic map.
Assume that the origin is an isolated fixed point of each iterate of $f$.
Then $\{\mathcal{N}_q(f)\}_{q=1}^{\infty}$, the sequence of the maximal number of periodic orbits of period $q$ that can be born from the fixed point zero under a small perturbation of $f$, is well defined. According to
Shub--Sullivan, Chow--Mallet--Paret--Yorke and G. Y. Zhang, the linear part of the holomorphic germ $f$ determines some natural restrictions on the sequence(cf. Theorem \ref{T1.2}).
Later, I. Gorbovickis proves that when the linear part of $f$ is contained in a certain large class
of diagonal matrices, it has no other restrictions on the sequence only when the dimension $n\leq2$
(cf. Theorem \ref{T1.3}). 
In this paper for the general case we obtain a sufficient and necessary condition that the linear part of
 $f$ has no other restrictions on the sequence $\{\mathcal{N}_q(f)\}_{q=1}^{\infty}$,
 except the ones given by Theorem \ref{T1.2}.
\end{abstract}
\section{Introduction and main results}
Let $\mathcal{O}(\mathbb{C}^{n},0,0)$ be the space of germs of all holomorphic maps from
$(\mathbb{C}^n,0)$ to $(\mathbb{C}^n,0)$. Assume that $h\in\mathcal{O}(\mathbb{C}^n,0,0)$ and
the origin is an isolated zero of $h$, that is, there exists an $n$-dimensional open ball $B$
centered at the origin such that $h$ is well defined in $B$ and $h^{-1}(0)\cap B=\{0\}$.
\emph{The zero order} of $h$ at the origin is $\pi _{h}(0):=\#(h^{-1}(v)\cap B)$,
where $v$ is a regular value of $h$ such that $|v|$ is small enough and $\#$ denotes the cardinality.
The number $\pi_{h}(0)$ is a well defined integer and is independent of the choice of the open ball $B$
(see \cite{LL} for the details). Let $h\in\mathcal{O}(\mathbb{C}^n,0,0)$ such that
the origin is an isolated fixed point of $h$, that is, an isolated zero of $h-id$,
where $id$ is the identity map. Then \emph{the fixed point index} of $h$ at the origin is
defined by $\mu_{h}(0):=\pi_{h-id}(0)$.

Let $f\in \mathcal{O}(\mathbb{C}^n,0,0)$ and $q$ a positive integer.
Assume that the origin is an isolated fixed point of $f^q$. In \cite{Dol}, A. Dold introduced the number
\[P_q(f,0):=\sum_{s\subset P(q)}(-1)^{\#s}\mu_{f^{q:s}}(0),\]
where $P(q)$ is the set of all prime factors of $q$ and $q:s=q(\prod_{k\in s}k)^{-1}$(If $s=\emptyset$, set $q:s=q$).
The number $P_q(f,0)$ is called \emph{the local Dold index} of $f$ at the origin and
the importance of the number is that it reveals the maximal number of periodic points of period $q$
that can be born from the origin under a small perturbation of $f$. More precisely,
let $U$ be an open subset of $\mathbb{C}^n$ with $0\in U$ such that $f^q$ is well defined in $U$.
In \cite{Zh3} G. Y. Zhang proved that for any ball $B\subset U$ centered at the origin such that $f^q$
has no fixed point in $B$ other than the origin, any $g$ has exactly $P_q(f,0)$
mutually distinct periodic points of period $q$ in $B$, provided that all fixed points of $g^q$ in $B$
are simple and that $g$ is sufficiently close to $f$. Equivalently,
$$P_q(f,0)=\limsup_{||g-f||\rightarrow0}\left|Per_q(g,B)\right|,$$
where $g$ is a holomorphic map from $U$ to $\mathbb{C}^n$, $||.||$ denotes the $\mathbb{C}{\rm-norm}$
in the space of holomorphic maps from $U$ to $\mathbb{C}^n$ and $\left|Per_q(g,B)\right|$ denotes the number of periodic points of period $q$ of $g$, completely contained in $B$.
In \cite{Zh5}, it is proved that $q$ divides $P_q(f,0)$ and
the number $\mathcal{N}_q(f):=\frac{P_q(f,0)}{q}$ is called the number of periodic orbits of $f$ of period $q$
hidden at the origin. (the reader
is referred to the references \cite{FL}, \cite{Gr1}--\cite{L}, \cite{Po1}--\cite{Po2}, \cite{Zh5}--\cite{Zh6}
for some interesting topics related to local and global Dold indices and the number of periodic orbits hidden at the fixed point.)

Let $f\in\mathcal{O}(\mathbb{C}^n,0,0)$ such that \[f(\bold{x})=\Lambda\bold{x}+o(\bold{x})\]
in a neighborhood of the origin, where $\Lambda$ is the linearization matrix of $f$ at the origin and
$o(\bold{x})$ denotes the higher order terms. There exists a close relationship between the sequence
$\{\mathcal{N}_q(f)\}^\infty_{q=1}$ and the periods of nonzero periodic points of the linear map
$\bold{x}\mapsto\Lambda\bold{x}$. Let $PE(\Lambda)$ be the set consisting of all periods of
nonzero periodic points of the linear map $\bold{x}\mapsto\Lambda\bold{x}$.
Zhang proved the following theorem (see \cite{Zh3}).

\begin{theorem}
\label{T1.2}Let $q$ be a positive integer and $f\in\mathcal{O}(\mathbb{C}^n,0,0)$ with
the linearization matrix  $\Lambda$ at the origin. Assume that the origin is an isolated fixed point of
both $f$ and $f^q$. Then
\begin{itemize}
\item for $q\ge 2$, the number $\mathcal{N}_q(f)>0$ if and only if $q\in PE(\Lambda)$;
\item for $q=1$, the number $\mathcal{N}_1(f) >1$ if and only if $1\in PE(\Lambda)$.
\end{itemize}
\end{theorem}
\begin{remark}
{\rm According to Shub--Sullivan \cite{SS} and Chow--Mallet-Paret--Yorke \cite{CMY},
a necessary condition so that there exists at least one periodic
orbit of period $q$ hidden at the origin, say, $\mathcal{N}_q(f)>0$
is that the linear part of $f$ at the origin has a
periodic point of period $q$. It is proved by Zhang in \cite{Zh3} that
the converse holds true.}
\end{remark}
In consideration of Theorem \ref{T1.2}, one may ask the following problem(see also \cite{Gor}).

\begin{problem} Given an $n\times n$ matrix $\Lambda$ and a non-negative integer sequence
$\{a_q\}^{\infty}_{q=1}$,  does there exist a holomorphic germ $f\in\mathcal{O}(\mathbb{C}^n,0,0)$
with the linearization matrix $\Lambda$ such that the origin is an isolated fixed point of
all iterates of $f$ and $\mathcal{N}_q(f)=a_q$ for $q=1,2,3,\cdots$?
\end{problem}

Given an $n\times n$ matrix $\Lambda$, the set $\mathcal{O}_{\Lambda}$
consists of $f\in\mathcal{O}(\mathbb{C}^n,0,0)$ such that the linearization matrix of $f$ is
$\Lambda$ and the origin is an isolated fixed point of all its iterates.
We say that a non-negative integer sequence $\{a_q\}^{\infty}_{q=1}$ is \emph{realizable}
for the matrix $\Lambda$, if there exists a germ $f\in\mathcal{O}_{\Lambda}$ such that
$\{\mathcal{N}_q(f)\}_{q=1}^{\infty}=\{a_q\}^{\infty}_{q=1}$ holds. It is easy to see that
a realizable sequence $\{a_q\}^{\infty}_{q=1}$ for $\Lambda$ must satisfy Zhang's conditions
(Theorem \ref{T1.2}). Thus we call a non-negative integer sequence $\{a_q\}^{\infty}_{q=1}$
\emph{admissible} for $\Lambda$, if it satisfies the conditions given by Theorem \ref{T1.2}.
That is, a non-negative integer sequence $\{a_q\}^{\infty}_{q=1}$ is admissible for $\Lambda$ if and only if
$$a_1\left\{\begin{matrix}\geq2\ ,&1\in PE(\Lambda)\\=1\ ,&1\not\in PE(\Lambda)\end{matrix}\right.\
{\rm and}\
a_q\left\{\begin{matrix}\geq1\ ,&q\in PE(\Lambda)\\=0\ ,&q\not\in PE(\Lambda)\end{matrix}\right.\
{\rm(}q\geq2\rm{)}.$$
The set of all admissible sequences $\{a_q\}^{\infty}_{q=1}$ associated to $\Lambda$ is denoted by
$AS(\Lambda)$, and if every sequence in $AS(\Lambda)$ is realizable, we will say that the matrix
$\Lambda$ is \emph{universal}.

The main purpose of this paper is to give a characteristic of the universal matrices.
We will prove the following theorem.

\begin{theorem}\label{T1}Let $\Lambda$ be an $n\times n$ matrix with
\begin{equation}
\label{F1.1}\Lambda=diag(A_{k_1},A_{k_2},...,A_{k_m}),
\end{equation}
where
\begin{equation*}
A_{k_j}=\begin{pmatrix}\lambda_j& 1& & & &\\
&\lambda_j &1 & & \\
& & \ddots&\ddots&\\
& &  &\lambda_j&1\\
& & & &\lambda_j
\end{pmatrix}_{k_j\times k_j},\ 1\le j\le m,\ k_1+k_2+\cdots+k_m=n
\end{equation*}
and $\lambda_j$ is the $d_j$-th primitive root of unity {\rm($j=1,2,\cdots,m$)}.
And let $M(\Lambda)$ denote the lowest common multiple of $d_1,d_2,\cdots,d_m$.
Then $\Lambda$ is universal if and only if one of the following two conditions holds:
\begin{enumerate}
\item By reordering the Jordan blocks of $\Lambda$ if it is necessary, the degrees satisfy
$1\le d_j\mid_{\neq}d_{j+1}$ ${\rm and}\ \lambda_j=\lambda^{\frac{d_{j+1}}{d_j}}_{j+1} {\rm(}j=1,2,\cdots,m-1{\rm)}$;
\item By reordering the Jordan blocks of $\Lambda$ if it is necessary, $\Lambda$ can be expressed as
$\Lambda=diag(\Lambda_1,\Lambda_2),$
where $\Lambda_1$ is of the form {\rm(\ref{F1.1})} and satisfies Condition (1), and
$\Lambda_2$ is a Jordan block with $M(\Lambda_2)\ge 2$ and $(M(\Lambda_1),M(\Lambda_2))=1$.
\end{enumerate}
Here, the notation $d_j\mid_{\neq}d_{j+1}$ means that $d_j$ divides $d_{j+1}$,
but $d_j\neq d_{j+1}$, and the notation
$(M(\Lambda_1),M(\Lambda_2))$ denotes the greatest common factor of $M(\Lambda_1)$ and $M(\Lambda_2)$.
\end{theorem}
\begin{remark}
{\rm Given a holomorphic germ $f\in\mathcal{O}_{\Lambda}$,
it is well known that biholomorphic transformations of coordinate do not change the fixed point indices of
all iterates of $f$. Furthermore, according to [\cite{Qiao}, Theorem 1.1 and Theorem 1.2] (or Theorem \ref{T3.1}
and Theorem \ref{T3} in Section \ref{Sec3}),
if the linearization matrix of $f$ is a Jordan matrix, then $f$ can be reduced to a holomorphic germ $g$ of
less than or equal dimension such that all eigenvalues of its linearization matrix are roots of unity and
$\mathcal{N}_q(g)$ is equal to $\mathcal{N}_q(f)$ ($q=1,2,3,\dots$). Therefore,
Theorem \ref{T1} is no loss of generality.}
\end{remark}
\begin{remark}
{\rm When an $n\times n$ matrix $\Lambda$ is diagonalizable, and all its eigenvalues are roots of unity of
pairwise relatively prime degrees greater than 1, I. Gorbovickis gave a sufficient and necessary condition
for the matrix $\Lambda$ to be universal, see \cite{Gor}. We restate the result of Gorbovickis as the following theorem.
\begin{theorem}[\protect{Gorbovickis}]
Let $\Lambda$ be of the form {\rm (\ref{F1.1})} such that $k_1=k_2=\cdots=k_m=1$ and $d_1,d_2,\cdots,d_m$
are greater than $1$ and pairwise relatively prime. Then $\Lambda$ is universal if and only if
$m\leq2$.\label{T1.3}
\end{theorem}
\noindent It is observed that when $\Lambda$ satisfies Theorem \ref{T1.3}, Condition $(1)$ in Theorem \ref{T1} degenerates to the case of $m=1$ in Theorem \ref{T1.3} and Condition $(2)$ in Theorem \ref{T1} degenerates to the case of $m=2$
in Theorem \ref{T1.3}. And Theorem \ref{T1.3} indicates that if $\Lambda$ satisfies Theorem \ref{T1.3},
then that $\Lambda$ is universal strongly depends on the dimension $n$ of the phase space.}
\end{remark}
\begin{remark}
{\rm In \cite{Gor}, to prove Theorem \ref{T1.3}, Gorbovickis established a method for reducing the number of periodic orbits of a special class of holomorphic germs hidden at the origin to the zero indices of a class of corresponding holomorphic germs at the origin. Then by analyzing the zero indices of corresponding holomorphic germs at the origin, he got the proof of Theorem \ref{T1.3}. In this paper, one of the difficulties in proving Theorem \ref{T1} is how to count the number of periodic orbits of general holomorphic germs hidden at the origin. Fortunately, \cite{Qiao} gave us a method to reduce the number of periodic orbits of general holomorphic germs hidden at the origin to zero indices of their corresponding holomorphic germs at the origin. But even then, how to analyze zero indices of corresponding holomorphic germs is still very complicated, for the holomorphic germs being considered are extremely general. To overcome this, we establish some monotonicity of the number of periodic orbits of holomorphic germs hidden at the origin with respect to their linearization matrices, which can reduce some complicated holomorphic germs to relatively simple ones. Then based on the monotonicity and Qiao's method \cite{Qiao} and Cronin's result (\cite{Cro} or \cite{Zh2}) estimating zero indices of holomorphic germs, 
we discuss six special cases including an extension of Gorbovickis' case. Combining these six cases, we finally get the proof of Theorem \ref{T1}.}
\end{remark}

\section{Several properties of zero indices and fixed point indices\label{Sec2}}
In this section we give several properties of zero indices and fixed point indices, which will be used later.

\begin{lemma}
\label{P1}For $g$, $\tilde{g}\in \mathcal{O}(\mathbb{C}^n,0,0)$, let
\begin{equation*}
g(x_1,x_2,\cdots,x_n)=(g_1,g_2,\cdots,g_n)
\end{equation*}
and
\begin{equation*}
\tilde{g}(x_1,x_2,\cdots,x_n)=(g_1,\dots,g_{j-1},\tilde{g}_j,g_{j+1},\dots,g_n).
\end{equation*}
If the origin is their isolated zero with multiplicity N and M respectively,
then the map
\begin{equation*}
f:(x_1,x_2,\cdots,x_n)\mapsto(g_1,\dots,g_{j-1},\tilde{g}_jg_j,g_{j+1},\dots,g_n)
\end{equation*}
has an isolated zero of multiplicity N+M at the origin.
\end{lemma}
\begin{remark}
Lemma \ref{P1} is well known in the theory of zero indices.
\end{remark}
\vspace{0.2cm}
Let $f, g\in \mathcal{O}(\mathbb{C}^n,0,0)$. Then $f$ and $g$ are \emph{algebraically equivalent} at the origin
if there exists a germ of a holomorphic family of linear nondegenerate maps $A(x)\in GL(n,\mathbb{C})$ such that
 $f(x)=A(x)g(x)$.
\begin{lemma}[\protect\cite{ARN}]
\label{PR1}Let $f, g\in \mathcal{O}(\mathbb{C}^n,0,0)$. If $f$ and $g$ are algebraically equivalent at the origin,
 then the origin is an isolated zero of $f$ if and only if it is an isolated zero of $g$.
 In this case, $\pi_f(0)=\pi_g(0)$.
\end{lemma}

\begin{lemma}[\protect\cite{Zh2}]
\label{P4}Let $f, g\in \mathcal{O}(\mathbb{C}^n,0,0)$. If the origin is an isolated zero of both $g$ and $f$ with
 multiplicity N and M respectively, then the composition $f\comp g$ of f and g has an isolated zero of
 multiplicity NM at the origin.
\end{lemma}
\begin{lemma}[\protect\cite{ARN}]
\label{P5}Let $f=(f_1,f_2,\cdots,f_n)\in\mathcal{O}(\mathbb{C}^n,0,0)$ be given by
\begin{equation*}
f_j(x_1,x_2,\cdots,x_n)=\sum_{k=1}^\infty f_{jk}(x_1,x_2,\cdots,x_n), j=1,\cdots,n
\end{equation*}
and $g=(g_1,g_2,\cdots,g_n)\in\mathcal{O}(\mathbb{C}^n,0,0)$ given by
\begin{equation*}
g_j(x_1,x_2,\cdots,x_n)=\sum_{k=1}^\infty g_{jk}(x_1,x_2,\cdots,x_n), j=1,\cdots,n,
\end{equation*}
where both $f_{jk}(x_1,x_2,\cdots,x_n)$ and $g_{jk}(x_1,x_2,\cdots,x_n)$ are homogeneous polynomials of degree $k$
in $x_1,$ $x_2,\cdots,x_n$. If the origin is an isolated zero of multiplicity $N>0$ of $f$ and $g_{jk}=f_{jk}$ for
$j=1,2,\cdots,n$ and $k=1,2,\cdots,N$, then the origin is also an isolated zero of multiplicity $N$ of $g$.
\end{lemma}

\begin{lemma}[Cronin, \protect\cite{Cro} or \cite{Zh2}]
\label{P2.6}Let $f=(f_{1},\dots ,f_{n})\in \mathcal{O}(\mathbb{C}^{n},0,0)$\ be given by
\begin{equation*}
f_{j}(x_{1},\dots ,x_{n})=\sum_{k=m_{j}}^{\infty }f_{jk}(x_{1},\cdots
,x_{n}),j=1,\cdots ,n,
\end{equation*}%
where each $m_j$ is a positive integer and each $f_{jk}$\ is a homogeneous polynomial of degree $k$\ in $%
x_{1},\dots ,x_{n}$. If $0$ is an isolated solution of the system of the $n$
equations
\begin{equation}
f_{jm_{j}}(x_{1},\dots ,x_{n})=0,j=1,\dots ,n,  \label{cr}
\end{equation}%
then $0$\ is an isolated zero of $f$ with
\begin{equation*}
\pi _{f}(0)=m_{1}\dots m_{n}.
\end{equation*}%
If $0$ is an isolated zero of $f$, but is not an isolated solution of the
system (\ref{cr}), then%
\begin{equation*}
\pi _{f}(0)>m_{1}\dots m_{n}.
\end{equation*}
\end{lemma}

\begin{lemma}[Shub and Sullivan, \cite{SS} or \cite{Zh2}]
\label{P8}Let $q>1$ be a positive integer and let $f\in
\mathcal{O}(\mathbb{C}^{n},0,0).$ If the origin is an isolated
fixed point of $f$ and for each eigenvalue $\lambda $ of the linearization matrix of $f$ at the origin,
either $\lambda=1$ or $\lambda ^{q}\neq 1$ holds, then the origin is also an
isolated fixed point of $f^{q}$ and $\mu_{f}(0)=\mu_{f^q}(0).$
\end{lemma}

\section{Counting the numbers of periodic orbits hidden at fixed points\label{Sec3}}
\subsection{Resonant polynomial formal norms\label{S1}}
Let $\Lambda$ be an $n\times n$ matrix with
\begin{equation}
\label{F02}\Lambda={\rm diag}(A_{k_1},A_{k_2},...,A_{k_m}),
\end{equation}
where
\begin{equation*}
A_{k_j}=\begin{pmatrix}\lambda_j& 1& & & &\\
&\lambda_j &1 & & \\
& & \ddots&\ddots&\\
& &  &\lambda_j&1\\
& & & &\lambda_j
\end{pmatrix}_{k_j\times k_j},\ 1\le j\le m,\ k_1+k_2+\cdots+k_m=n
\end{equation*}
and $\lambda_j$ is the $d_j$-th primitive root of unity {\rm($j=1,2,\cdots,m$)} and $d_j$ is set to $0$ if $\lambda_j$ is not a root of unity. Denote by $e_1,\cdots,e_n$ the standard orthonormal basis in
$\mathbb{C}^n$. A monomial of degree greater than $1$ proportional to the monomial $x_1^{i_1}\cdots x_n^{i_n}e_s$
is said to be \emph{resonant with respect to} $\Lambda$, if the eigenvalue
$$\lambda_s=\lambda_1^{i_1}\cdots\lambda_n^{i_n}.$$
A germ $f\in\mathcal{O}(\mathbb{C}^n,0,0)$ is said to be \emph{in resonant polynomial normal form},
if $f(x)=\Lambda x+F(x)$, where $F(x)$ is a sum of finitely many resonant monomials with respect to $\Lambda$.
The set of all resonant polynomial normal forms with $\Lambda$ their linearization matrices at the origin in $\mathcal{O}(\mathbb{C}^n,0,0)$ is denoted by $\mathcal{RE}_{\Lambda}$ and the set of all resonant polynomial normal forms in $\mathcal{O}_{\Lambda}$ is denoted by $\mathcal{R}_{\Lambda}$.

The following theorem in \cite{Gor} tells us that every holomorphic germ in $\mathcal{O}_{\Lambda}$
can be reduced to a corresponding resonant polynomial normal form in terms of the numbers of period orbits
hidden at the origin.
\begin{theorem}[\protect{Gorbovickis}]
\label{T3.1}Assume the matrix $\Lambda$ is of the form $(\ref{F02})$. For every germ $f\in\mathcal{O}_{\Lambda}$,
there exists a germ $\tilde f\in\mathcal{R}_{\Lambda}$ such that
$\mathcal{N}_q(f)=\mathcal{N}_q(\tilde f)$ holds for all $q\geq1$.
\end{theorem}
\begin{remark}
{\rm Gorbovickis' theorem is stated for diagonal matrices but the adaptation to $(\ref{F02})$ is straightforward.}
\end{remark}

\subsection{Counting the numbers of periodic orbits hidden at fixed points for general case\label{S2}}
\quad In this subsection, we will give a method to count the numbers of periodic orbits hidden at the origin
for resonant polynomial normal forms with the origin as an isolated fixed point of all their iterates.
The method comes from \cite{Qiao}. For convenience, we first introduce some notations as follows.

Let $W_n$ be the set of all words of 0's and 1's of length $n$ and $W_n^*=W_n\backslash\{(0\cdots0)\}$.
For every subset $S$ of the set $\{1,2,\cdots,n\}$, we set $W(S)=(w_1\cdots w_n)\in W_n$,
where $w_j=1$ if and only if $j\in S$. Similarly, given $w=(w_1\cdots w_n)\in W_n$,
we denote by $S(w)$ the set of all indices $j$ such that $w_j=1$. If $w\in W_n^*$ and
$S(w)=\{j_1,\cdots,j_{|w|}\}$, where $j_1<\cdots<j_{|w|}$ and $|w|=\sum_{j=1}^n w_j$,
the subspace of $\mathbb{C}^n$ spanned by the coordinates with indices from $S(w)$ will be denoted by
$\mathbb{C}^{|w|}$. There is a natural projection

$$p_w:\mathbb{C}^n\longrightarrow\mathbb{C}^{|w|},\ (x_1,\cdots,x_n)\mapsto(x_{j_1},\cdots,x_{j_{|w|}})$$
and a natural inclusion
$$i_w:\mathbb{C}^{|w|}\longrightarrow\mathbb{C}^n,\ (\tilde x_{j_1},\cdots,\tilde x_{j_{|w|}})\mapsto
(\tilde x_1,\cdots,\tilde x_n),$$
where $\tilde x_j=0$ if $j\not\in S(w)$.

Let $\Lambda$ be of the form (\ref{F02}) (or (\ref{F1.1})).
We set $s_0=0$ and $s_j=k_1+\cdots+k_j$ for $j=1,2,\cdots,m$.
We define a map $\theta$ :$\{1,2,\cdots,n\}\longrightarrow\{1,2,\cdots,m\}$ such that $\theta(j)=q$
if $s_{q-1}<j\leq s_q$. we set $w(\Lambda)=(w_1w_2\cdots w_n)\in W_n$ with $w_j=0$ if and only if $d(j)=0$.
For any positive integer $l$, we set
$w(\Lambda, l)=(w_{1l}w_{2l}\cdots w_{nl})\in W_n$ with $w_{jl}=1$ if and only if $d_{\theta(j)}\mid l$.
The notation $w(\Lambda, l)$ is also written as $w(l)$ when no ambiguity is caused. We define a map
$$\tau: \mathcal{RE}_{\Lambda}\rightarrow \mathcal{O}(\mathbb{C}^n,0,0),\ f(\bold{x})=\Lambda\bold{x}+F(\bold{x})
\mapsto \tau f(\bold{x})=(\Lambda-\tilde\Lambda)\bold{x}+F(\bold{x}),$$
where the entries of the main diagonal of $\tilde\Lambda$ are identical to that of $\Lambda$ and
the other entries are zeros. Note that the notation $\tilde\Lambda$ will be also used later.

The following theorem in \cite{Qiao} gives an efficient method to count fixed point indices of all iterates of
a holomorphic germ at the origin.
\begin{theorem}[\protect{Qiao, Qu and Zhang}]
\label{T3}Let $\Lambda$ be a matrix of the form {\rm(\ref{F02})} and $f$ be a germ in $\mathcal{RE}_{\Lambda}$.
Then the origin in $\mathbb{C}^{|w(\Lambda)|}$ is an isolated zero of
$p_{w(\Lambda)}\comp\tau f\comp i_{w(\Lambda)}$ if and only if the origin in
$\mathbb{C}^n$ is an isolated fixed point of $f^k$ for any $k\geq 1$. Moreover,
in this case we have for any $k\geq 1$
\[\mu_{f^k}(0)=\pi_{p_{w(k)}\comp\tau f\comp i_{w(k)}}(0).\]
\end{theorem}

If $\Lambda$ is of the form (\ref{F1.1}) and $f\in\mathcal{R}_{\Lambda}$, then, by Theorem \ref{T3} and [\cite{Zh3}, Lemma 2.7], we have
\begin{equation}
\label{F3.7}\pi_{p_{w(d)}\comp\tau f\comp i_{w(d)}}(0)=\sum\limits_{q\mid d,q\in PE(\Lambda)\cup
\{1\}}q\mathcal{N}_q(f)
\end{equation}
for all $d\in PE(\Lambda)$. We can solve out $\mathcal{N}_q(f)$ for every $q\in PE(\Lambda)$ from the above equations,
and hence this gives us a general method to compute the numbers of periodic orbits hidden at the origin.
At the end of this subsection, we use this general method to compute the following example.
\begin{example}
\label{EX3.1}\rm{Let $\Lambda$ be of the form (\ref{F1.1}) with $m\geq2$ such that $d_j\mid_{\not=} d_{j+1}$ and
$\lambda_j=\lambda_{j+1}^{\frac{d_{j+1}}{d_j}}$ for $j=1,\cdots,m-2$, and $(d_{m-1},d_m)=1$ and $d_m>1$.

When $d_1>1$, let $g\in\mathcal{R}_{\Lambda}$ with $g(x_1,\cdots,x_n)=(g_1,g_2,\cdots,g_n)$ such that
\begin{align*}
g_{s_1}&=\lambda_1x_{s_1}+x_{s_0+1}^{r_1d_1+1}-x_{s_0+1}x_{s_{m-1}+1}^{r_mr_1d_m}+
x_{s_0+1}x_{s_{m-1}+1}^{r_{1m}d_m}+x_{s_1+1}^{\frac{d_2}{d_1}},\\
g_{s_{m-1}}&=\lambda_{m-1}x_{s_{m-1}+1}+x_{s_{m-2}+1}(x_{s_0+1}^{r_{m-1}d_1}
-x_{s_{m-1}+1}^{r_md_m}x_{s_0+1}^{(r_{m-1}-1)d_1}+x_{s_{m-1}+1}^{r_{(m-1)m}d_m}),\\
g_{s_m}&=\lambda_mx_{s_m}+x_{s_{m-1}+1}(x_{s_0+1}^{d_1}-x_{s_{m-1}+1}^{r_md_m}),\\
g_{s_t}&=\lambda_tx_{s_t}+x_{s_{t-1}+1}(x_{s_0+1}^{r_td_1}
-x_{s_{m-1}+1}^{r_md_m}x_{s_0+1}^{(r_t-1)d_1}
+x_{s_{m-1}+1}^{r_{tm}d_m})+x_{s_t+1}^{\frac{d_{t+1}}{d_t}}
\end{align*}
and
$$g_j=\lambda_{\theta(j)}x_j+x_{j+1}$$
for $t=2,\cdots,m-2$ and $j\in\{1,\cdots,n\}\setminus\{s_1,\cdots,s_m\}$.

By Lemma \ref{P1} and Lemma \ref{PR1}, we have that
$$\pi_{p_{w(d_m)}\comp\tau g\comp i_{w(d_m)}}(0)=r_md_m+1,$$
$$\pi_{p_{w(d_j)}\comp\tau g\comp i_{w(d_j)}}(0)=\sum_{s=1}^{j}r_sd_s+1$$
and
$$\pi_{p_{w(d_jd_m)}\comp\tau g\comp i_{w(d_jd_m)}}(0)=\sum_{s=1}^{j}r_sd_s+\sum_{s=1}^{j}r_{sm}d_sd_m+r_md_m+1$$
for $j=1,\cdots,m-1$.
Then
$$r_md_m=d_m\mathcal{N}_{d_m}(f),$$
$$\sum_{s=1}^{j}r_sd_s=\sum_{s=1}^{j}d_s\mathcal{N}_{d_s}(f)$$
and
$$\sum_{s=1}^{j}r_sd_s+\sum_{s=1}^{j}r_{jm}d_jd_m+r_md_m=\sum_{s=1}^{j}d_s\mathcal{N}_{d_s}(f)+
\sum_{s=1}^{j}d_sd_m\mathcal{N}_{d_sd_m}(f)+d_m\mathcal{N}_{d_m}(f)$$
for $j=1,\cdots,m-1$.
Combining these equations, we have that $\mathcal{N}_{d_j}(g)=r_j$ for $j=1,2,3,\cdots,m-1,m$ and
$\mathcal{N}_{d_jd_m}(g)=r_{jm}$ for $j=1,2,3,\cdots,m-1$.

When $d_1=1$, let $g\in\mathcal{R}_{\Lambda}$ with $g(x_1,\cdots,x_n)=(g_1,g_2,\cdots,g_n)$ such that
\begin{align*}
g_{s_1}&=\lambda_1x_{s_1}+x_{s_0+1}^{r_1d_1+1}+
x_{s_{m-1}+1}^{r_md_m}+x_{s_1+1}^{\frac{d_2}{d_1}},\\
g_{s_{m-1}}&=\lambda_{m-1}x_{s_{m-1}+1}+x_{s_{m-2}+1}(x_{s_0+1}^{r_{m-1}d_1}
+x_{s_{m-1}+1}^{r_{(m-1)m}d_m}),\\
g_{s_m}&=\lambda_{m}x_{s_{m}}+x_{s_{m-1}+1}x_{s_0+1}^{d_1},\\
g_{s_t}&=\lambda_tx_{s_t}+x_{s_{t-1}+1}(x_{s_0+1}^{r_td_1}
+x_{s_{m-1}+1}^{r_{tm}d_m})+x_{s_t+1}^{\frac{d_{t+1}}{d_t}}
\end{align*}
and
$$g_j=\lambda_{\theta(j)}x_j+x_{j+1}$$
for $t=2,\cdots,m-2$ and $j\in\{1,\cdots,n\}\setminus\{s_1,\cdots,s_m\}$.

Similarly, we have that $\mathcal{N}_{d_j}(g)=r_j$ for $j=2,3,\cdots,m-1,m$, $\mathcal{N}_{d_jd_m}(g)=r_{jm}$ for
$j=2,3,\cdots,m-1$ and $\mathcal{N}_{d_1}(g)=r_1+1$.}

\end{example}

\subsection{Counting the numbers of periodic orbits hidden at fixed points for a special case\label{S3}}
Let $\Lambda$ be of the form (\ref{F1.1}). If we add appropriate restrictions to $\Lambda$,
it is possible to get a more effective way to count the numbers of periodic orbits of holomorphic germs
hidden at the origin.

Assume that there exist $j_1,j_2,\cdots,j_t\in\{1,2,\cdots,m\}$ ($1\leq t\leq m$) such that
\begin{equation}
\label{F3.1} [d_{j_1},\cdots,d_{j_t}]=M(\Lambda)\ {\rm and}\ d_{j_s}\nmid [d_1,\cdots,d_{j_s-1},d_{j_s+1},\cdots,d_m]
\end{equation}
for $s=1,2,\cdots,t$, where $[d_{j_1},\cdots,d_{j_t}]$ denotes the lowest common multiple of
$d_{j_1},\cdots,d_{j_t}$ and $[d_1,\cdots,d_{j_s-1},d_{j_s+1},\cdots,d_m]$ denotes the lowest common multiple of
$d_1,\cdots,d_{j_s-1},d_{j_s+1},\cdots,d_m$. Let $f\in\mathcal{R}_{\Lambda}$. In this subsection we will give a more efficient method to
count the numbers of periodic orbits of $f$ hidden at the origin. This method is a natural extension of
[\cite{Gor}, Theorem 3.4] or [\cite{Qiao}, Corollary 1.2].

Let $\tau f=(f_1,f_2,\cdots,f_n)$. Then
\begin{equation}
\label{F3.5}f_j=x_{j+1}+{\rm higher\ terms}
\end{equation}
for all $j\in\{1,2,\cdots,n\}\setminus\{s_1,\cdots,s_m\}$.
First, we consider the special case of $f$ that all variables of every monomial of $f_{s_1},\cdots,f_{s_m}$ come from
$x_{s_0+1},x_{s_1+1},\cdots,x_{s_{m-1}+1}$. For any $s\in\{1,2,\cdots,t\}$, by (\ref{F3.1}), we have
$x_{s_{j_s-1}+1}\mid f_{s_{j_s}}$. Let $\tilde\tau f=(g_1,g_2,\cdots,g_n)$ with
$g_{s_{j_s}}=\frac{f_{s_{j_s}}}{x_{s_{j_s-1}+1}}$ for $s\in \{1,2,\cdots,t\}$ and $g_j=f_j$ for
$j\in\{1,\cdots,n\}\setminus\{s_{j_1},\cdots,s_{j_t}\}$. It is easy to see that
the origin is an isolated zero of $\tilde\tau f$. Furthermore, we have the following proposition,
whose proof is a direct extension of the proof of [\cite{Qiao}, Corollary 1.2].
\begin{proposition}
\label{L2}$\mathcal{N}_{M(\Lambda)}(f)=\frac{\pi_{\tilde\tau f}(0)}{M(\Lambda)}$.
\end{proposition}
For any positive integer $q\in PE(\Lambda)$, the matrix $p_{w(q)}\comp\Lambda\comp i_{w(q)}$
has also the form (\ref{F1.1}), and if it satisfies (\ref{F3.1}), we can obtain
$\tilde\tau(p_{w(q)}\comp f\comp i_{w(q)})$ from $\tau(p_{w(q)}\comp f\comp i_{w(q)})$,
as $\tilde\tau f$ is obtained from $\tau f$. Then by Proposition \ref{L2}, we have
\begin{equation*}
\mathcal{N}_q(p_{w(q)}\comp f\comp i_{w(q)})=\frac{\pi_{\tilde\tau(p_{w(q)}\comp f\comp
i_{w(q)})}(0)}{q}.
\end{equation*}
According to Theorem \ref{T3}, it is easy to see that
\begin{equation*}
\mathcal{N}_q(f)=\mathcal{N}_q(p_{w(q)}\comp f\comp i_{w(q)}).
\end{equation*}
Thus we obtain that
\begin{equation}
\mathcal{N}_q(f)=\frac{\pi_{\tilde\tau(p_{w(q)}\comp f\comp
i_{w(q)})}(0)}{q}.\label{FT}
\end{equation}
The formula (\ref{FT}) gives us an efficient method to count the numbers of periodic orbits
hidden at the origin for the above special case (Interested readers can also use Proposition \ref{L2}
to compute Example \ref{EX3.1}).

The following proposition tells us that the general case of $f$ can be reduced to this special case.

\begin{proposition}
There exists $\tilde f\in\mathcal{R}_{\Lambda}$ with $\tau\tilde f=(\tilde f_1,\tilde f_2,\cdots,\tilde f_n)$
such that $\mathcal{N}_q(\tilde f)=\mathcal{N}_q(f)$ for all $q\geq1$ and
all variables of every monomial of $\tilde f_{s_1},\cdots,\tilde f_{s_m}$ come from
$x_{s_0+1},x_{s_1+1},\cdots,x_{s_{m-1}+1}$. Moreover, we can further require that
$\tilde f_j=f_j$ for all $j\in\{1,2,\cdots,n\}\setminus\{s_1,\cdots,s_m\}$.\label{p3.3}
\end{proposition}
\begin{proof}
Let
$$f_{s_j}(x_1,\cdots,x_n)=\sum_{i_1+\cdots+i_n=2}^{\infty}c_{i_1\cdots i_n}^{(j)}x_1^{i_1}\cdots x_n^{i_n}$$
and $c_{i_1\cdots i_n}^{(j)}$, $i_1+\cdots+i_n=2,3,\cdots$ are complex numbers with finitely many nonzero terms
for $j=1,\cdots,m$. Let
$$f_{s_j}^{(1)}(x_1,\cdots,x_n)=f_{s_j}(x_1,\cdots,x_n)
-\sum_{\substack{i_1+i_2+\cdots+i_n=2,\\i_{s_0+1}+i_{s_1+1}+\cdots+i_{s_{m-1}+1}\leq1}}c_{i_1\cdots i_n}^{(j)}
\frac{\prod_{j=1}^{n-1}f_j^{i_{j+1}}}{\prod_{j=1}^{m-1}f_{s_j}^{i_{s_j+1}}}\prod_{j=1}^{m}x_{s_{j-1}+1}^{i_{s_{j-1}+1}}$$
for $j=1,\cdots,m$. By (\ref{F3.5}), we can get that all variables of monomials of degree $2$ of
$f_{s_1}^{(1)},\cdots,f_{s_m}^{(1)}$ come from $x_{s_0+1},x_{s_1+1},\cdots,x_{s_{m-1}+1}$.
Since $f_s\comp\tilde\Lambda=\lambda_{\theta(s)}f_s$ for $s=1,2,\cdots,n$, we have that
$f_{s_j}^{(1)}\comp\tilde\Lambda=\lambda_{\theta(s_j)}f_{s_j}^{(1)}$ for $j=1,\cdots,m$.
Let
$$f^{(1)}(\bold{x})=\tilde\Lambda\bold{x}+(f_1^{(1)}(\bold{x}),\cdots,f_n^{(1)}(\bold{x}))$$
with $f_j^{(1)}(\bold{x})=f_{j}(\bold{x})$
for all $j\in\{1,2,\cdots,n\}\setminus\{s_1,\cdots,s_m\}$. Then $f^{(1)}\comp\tilde\Lambda=\tilde\Lambda f^{(1)}$, that is, $f^{(1)}$ is in resonant polynomial formal norm. Combining Lemma \ref{PR1} and the definition of $f^{(1)}$, we see
that
\begin{equation}
\label{F3.9}\pi_{p_{w(q)}\comp\tau f\comp i_{w(q)}}(0)=\pi_{p_{w(q)}\comp \tau f^{(1)}\comp
i_{w(q)}}(0)
\end{equation}
for all $q\geq1$. Together with Theorem \ref{T3}, we have that $f^{(1)}\in\mathcal{R}_{\Lambda}$ and
\begin{equation*}
\mathcal{N}_q(f^{(1)})=\mathcal{N}_q(f)
\end{equation*}
for all $q\geq1$.

Similarly, let
$$f_{s_j}^{(1)}(x_1,\cdots,x_n)=\sum_{i_1+\cdots+i_n=2}^{\infty}c_{i_1\cdots i_n}^{(1j)}x_1^{i_1}
\cdots x_n^{i_n}$$
and $c_{i_1\cdots i_n}^{(1j)}$, $i_1+\cdots+i_n=2,3,\cdots$ are complex numbers with finitely many nonzero terms for $j=1,2,\cdots,m$. Let
$$f_{s_j}^{(2)}(x_1,\cdots,x_n)=f_{s_j}^{(1)}(x_1,\cdots,x_n)
-\sum_{\substack{i_1+i_2+\cdots+i_n=3,\\i_{s_0+1}+i_{s_1+1}+\cdots+i_{s_{m-1}+1}\leq2}}c_{i_1\cdots i_n}^{(1j)}
\frac{\prod_{j=1}^{n-1}f_j^{i_{j+1}}}{\prod_{j=1}^{m-1}f_{s_j}^{i_{s_j+1}}}\prod_{j=1}^{m}x_{s_{j-1}+1}^{i_{s_{j-1}+1}}$$
for $j=1,\cdots,m$. By (\ref{F3.5}), we can get that all variables of monomials of degree less than $4$ of
$f_{s_1}^{(2)},\cdots,f_{s_m}^{(2)}$ come from $x_{s_0+1},x_{s_1+1},\cdots,x_{s_{m-1}+1}$.
It is easy to see that $f_{s_j}^{(2)}\comp\tilde\Lambda=\lambda_{\theta(s_j)}f_{s_j}^{(2)}$ for
$j=1,\cdots,m$. Let
$$f^{(2)}(\bold{x})=\tilde\Lambda\bold{x}+(f_1^{(2)}(\bold{x}),\cdots,f_n^{(2)}(\bold{x}))$$
with
$f_j^{(2)}(\bold{x})=f_{j}(\bold{x})$ for all $j\in\{1,2,\cdots,n\}\setminus\{s_1,\cdots,s_m\}$. Then
$f^{(2)}\comp\tilde\Lambda=\tilde\Lambda f^{(2)}$, that is, $f^{(2)}$ is in resonant polynomial normal form.
According to Lemma \ref{PR1}, the definition of $f^{(2)}$ and (\ref{F3.9}), we have
\begin{equation}
\pi_{p_{w(q)}\comp\tau f\comp i_{w(q)}}(0)=\pi_{p_{w(q)}\comp \tau f^{(2)}\comp
i_{w(q)}}(0)
\end{equation}
for all $q\geq1$. Together with Theorem \ref{T3}, we have that $f^{(2)}\in\mathcal{R}_{\Lambda}$ and
\begin{equation*}
\mathcal{N}_q(f^{(2)})=\mathcal{N}_q(f)
\end{equation*}
for all $q\geq1$.

Similarly, for all $l\geq1$, there exists $f^{(l)}\in\mathcal{R}_{\Lambda}$ with
$\tau f^{(l)}=(f_1^{(l)},\cdots,f_n^{(l)})$ such that all variables of monomials of degree less than $l+2$ of
$f_{s_1}^{(l)},\cdots,f_{s_m}^{(l)}$ come from $x_{s_0+1},x_{s_1+1},\cdots,x_{s_{m-1}+1}$ and
$f_j^{(l)}=f_{j}$ for all $j\in\{1,2,\cdots,n\}\setminus\{s_1,\cdots,s_m\}$,
and at the same time
\begin{equation}
\label{F3.8}\pi_{p_{w(q)}\comp\tau f\comp i_{w(q)}}(0)=\pi_{p_{w(q)}\comp \tau f^{(l)}\comp i_{w(q)}}(0),
\end{equation}
$f^{(l)}\in\mathcal{R}_{\Lambda}\ {\rm and}\ \mathcal{N}_q(f^{(l)})=\mathcal{N}_q(f)$
for all $q\geq1$.

We set
$$l_{max}:=\max\limits_{l\geq1}\pi_{p_{w(l)}\comp \tau f\comp i_{w(l)}}(0).$$
By (\ref{F3.8}), we have
\begin{equation}
\label{F3.10}l_{max}=\max\limits_{l\geq1}\pi_{p_{w(l)}\comp \tau f^{(l_{max})}\comp i_{w(l)}}(0).
\end{equation}
Let $\tilde f\in\mathcal{R}_{\Lambda}$
with $\tau\tilde f=(\tilde f_1,\tilde f_2,\cdots,\tilde f_n)$
such that
\begin{itemize}
\item every monomial of $f_{s_j}^{(l_{max})}-\tilde f_{s_j}$
has degree greater than $l_{max}$ for $j=1,\cdots,m$,
\item $\tilde f_{s_j}$ has no monomial of degree greater than $l_{max}$ for $j=1,\cdots,m$,
\item $\tilde f_j=f_{j}$ for all $j\in\{1,2,\cdots,n\}\setminus\{s_1,\cdots,s_m\}$.
\end{itemize}
By Lemma \ref{P5} and (\ref{F3.10}), we have that for all $q\geq1$,
$$\pi_{p_{w(q)}\comp\tau\tilde f\comp i_{w(q)}}(0)=\pi_{p_{w(q)}\comp \tau f^{(l_{max})}\comp
i_{w(q)}}(0).$$
Together with (\ref{F3.8}), we have that for all $q\geq1$,
$$\pi_{p_{w(q)}\comp\tau\tilde f\comp i_{w(q)}}(0)=\pi_{p_{w(q)}\comp \tau f\comp i_{w(q)}}(0).$$
Together with Theorem \ref{T3}, we have that $\tilde f\in\mathcal{R}_{\Lambda}$ and
\begin{equation*}
\mathcal{N}_q(\tilde f)=\mathcal{N}_q(f)
\end{equation*}
for all $q\geq1$. Thus $\tilde f$ is exactly what we want.
\end{proof}


At last, based on the method of this subsection, we give an extensional version of Theorem \ref{T1.3},
which will be needed in the proof of Theorem \ref{T1} in Section \ref{s5}.

\begin{proposition}
Let $\Lambda$ be of the form {\rm(\ref{F1.1})} such that $d_1,d_2,\cdots,d_m$ are pairwise relatively prime and
at most one of them is $1$. Then $\Lambda$ is universal if and only if $m\leq3$ and one of $d_1$, $d_2$ and
$d_3$ is $1$
if $m=3$.\label{Cr3.2}
\end{proposition}
\begin{proof}
Let $f\in\mathcal{R}_{\Lambda}$ and $\tau f=(f_1,f_2,\cdots,f_n)$. By Proposition \ref{p3.3},
we can suppose that all variables of every monomial of $f_{s_1},\cdots,f_{s_m}$ come from
$x_{s_0+1},x_{s_1+1},\cdots,x_{s_{m-1}+1}$.

We first consider the case that $d_1,d_2,\cdots,d_m$ are greater than $1$. Since $d_1,d_2,\cdots,d_m$
are pairwise relatively prime, we have for every $j\in\{1,2,\cdots,m\}$,
\begin{equation*}
f_{s_j}(x_1,\cdots,x_n)=x_{s_{j-1}+1}(a_{j1}x_{s_0+1}^{d_1}+\cdots+a_{jm}x_{s_{m-1}+1}^{d_m}+P_j(x_{s_0+1}^{d_1},
\cdots,x_{s_{m-1}+1}^{d_m})),
\end{equation*}
where $a_{j1},\cdots,a_{jm}$ are complex numbers and $P_j$ is a polynomial in
$x_{s_0+1}^{d_1},\cdots,x_{s_{m-1}+1}^{d_m}$ without constant and linear terms.
Let $g=(g_1,g_2,\cdots,g_n)$ such that 
\begin{equation}
\label{F3.3}g_{s_j}(x_1,\cdots,x_n)=\frac{f_{s_j}(x_1,\cdots,x_n)}{x_{s_{j-1}+1}}=
a_{j1}x_{s_0+1}^{d_1}+\cdots+a_{jm}x_{s_{m-1}+1}^{d_m}+P_j(x_{s_0+1}^{d_1},
\cdots,x_{s_{m-1}+1}^{d_m})
\end{equation}
for $j=1,2,\cdots,m$ and
\begin{equation}
\label{F3.4}g_j(x_1,\cdots,x_n)=f_j(x_1,\cdots,x_n)=x_{j+1}+{\rm higher\ terms}
\end{equation}
for all $j\in\{1,2,\cdots,n\}\setminus\{s_1,\cdots,s_m\}$.
By (\ref{FT}), we have
\begin{equation}
\label{F3.2}\mathcal{N}_q(f)=\frac{\pi_{p_{w(q)}\comp g\comp i_{w(q)}}(0)}{q}
\end{equation}
for all $q\in PE(\Lambda)$.

We claim that $\mathcal{N}_{d_1}(f)>1$, $\mathcal{N}_{d_2}(f)>1$, $\mathcal{N}_{d_3}(f)>1$,
$\mathcal{N}_{d_1d_2}(f)=1$, $\mathcal{N}_{d_1d_3}(f)=1$ and $\mathcal{N}_{d_2d_3}(f)=2$ imply
$\mathcal{N}_{d_1d_2d_3}(f)=1$, as in the proof of [\cite{Gor}, Theorem 1.7].

Indeed, by (\ref{F3.2}), we have that $\mathcal{N}_{d_1}(f)>1$ implies
$\pi_{p_{w(d_1)}\comp g\comp i_{w(d_1)}}(0)>d_1$. Together with (\ref{F3.3}),(\ref{F3.4}) and
Lemma \ref{P2.6}, we get that $a_{11}=0$. Similarly, we have that $\mathcal{N}_{d_2}(f)>1$ implies $a_{22}=0$ and
$\mathcal{N}_{d_3}(f)>1$ implies $a_{33}=0$. By (\ref{F3.2}), we see that $\mathcal{N}_{d_1d_2}(f)=1$ implies
$\pi_{p_{w(d_1d_2)}\comp g\comp i_{w(d_1d_2)}}(0)=d_1d_2$. Then
$$\pi_{p_{w(d_1d_2)}\comp g\comp i_{w(d_1d_2)}\comp u}(0)=d_1^2d_2^2,$$
where $$u(x_1,\cdots,x_{s_1},x_{s_1+1},\cdots,x_{s_2})=(x_1^{d_2},x_2,\cdots,x_{s_1},x_{s_1+1}^{d_1},x_{s_1+2},
\cdots,x_{s_2}).$$
Together with (\ref{F3.3}), (\ref{F3.4}) and Lemma \ref{P2.6}, we get that $a_{11}a_{22}-a_{12}a_{21}\not=0$.
Again $a_{11}=a_{22}=0$, we have $a_{12}a_{21}\not=0$. Similarly, $\mathcal{N}_{d_1d_3}(f)=1$ implies
$a_{13}a_{31}\not=0$, and $\mathcal{N}_{d_2d_3}(f)=2$ implies $a_{23}\not=a_{32}=0$ or $a_{32}\not=a_{23}=0$.
According to the above results, by Lemma \ref{P2.6} we can get that
$$\pi_{p_{w(d_1d_2d_3)}\comp g\comp i_{w(d_1d_2d_3)}\comp uu}(0)=d_1^3d_2^3d_3^3,$$
where
$$uu(x_1,\cdots,x_{s_1},x_{s_1+1},\cdots,x_{s_2},x_{s_2+1},\cdots,x_{s_3})$$
$$=(x_1^{d_2d_3},x_2,\cdots,x_{s_1},x_{s_1+1}^{d_1d_3},x_{s_1+2},\cdots,x_{s_2},x_{s_2+1}^{d_1d_2},x_{s_2+2},
\cdots,x_{s_3}),$$
and hence $\mathcal{N}_{d_1d_2d_3}(f)=1$. Thus we complete the proof of the claim.

Next we assume that one of $d_1,d_2,\cdots,d_m$ is $1$. Without loss of generality, let $d_1=1$. Then we have
\begin{equation*}
f_{s_1}(x_1,\cdots,x_n)=a_{11}x_{s_0+1}^{d_1}+\cdots+a_{1m}x_{s_{m-1}+1}^{d_m}+P_1(x_{s_0+1}^{d_1},
\cdots,x_{s_{m-1}+1}^{d_m})
\end{equation*}
and
\begin{equation*}
f_{s_j}(x_1,\cdots,x_n)=x_{s_{j-1}+1}(a_{j1}x_{s_0+1}^{d_1}+\cdots+a_{jm}x_{s_{m-1}+1}^{d_m}+P_j(x_{s_0+1}^{d_1},
\cdots,x_{s_{m-1}+1}^{d_m}))
\end{equation*}
for $j=2,\cdots,m$, where each $a_{jk}$ is a complex number and each $P_j$ is a polynomial in
$x_{s_0+1}^{d_1},\cdots,x_{s_{m-1}+1}^{d_m}$ without constant and linear terms.
Let $g=(g_1,g_2,\cdots,g_n)$ such that

\begin{equation}
\label{F31}g_{s_1}(x_1,\cdots,x_n)=f_{s_1}(x_1,\cdots,x_n)=a_{11}x_{s_0+1}^{d_1}+\cdots+a_{1m}
x_{s_{m-1}+1}^{d_m}+P_1(x_{s_0+1}^{d_1},
\cdots,x_{s_{m-1}+1}^{d_m})
\end{equation}
and
\begin{equation}
\label{F33} g_{s_j}(x_1,\cdots,x_n)=\frac{f_{s_j}(x_1,\cdots,x_n)}{x_{s_{j-1}+1}}=a_{j1}x_{s_0+1}^{d_1}+
\cdots+a_{jm}x_{s_{m-1}+1}^{d_m}+P_j(x_{s_0+1}^{d_1},
\cdots,x_{s_{m-1}+1}^{d_m})
\end{equation}
for $j=2,\cdots,m$. Moreover,
\begin{equation}
\label{F34}g_j(x_1,\cdots,x_n)=f_j(x_1,\cdots,x_n)=x_{j+1}+{\rm higher\ terms}
\end{equation}
for all $j\in\{1,2,\cdots,n\}\setminus\{s_1,\cdots,s_m\}$.
By (\ref{FT}), we have
\begin{equation}
\label{F32}\mathcal{N}_l(f)=\frac{\pi_{p_{w(l)}\comp g\comp i_{w(l)}}(0)}{l}
\end{equation}
for all $l\in PE(\Lambda)$.

We claim that $\mathcal{N}_{d_2}(f)\geq2$, $\mathcal{N}_{d_3}(f)\geq2$ and $\mathcal{N}_{d_4}(f)\geq2$ imply
$\mathcal{N}_{d_2d_3}(f)\geq2$ or $\mathcal{N}_{d_2d_4}(f)\geq2$ or $\mathcal{N}_{d_3d_4}(f)\geq2$.
According to the definition of $g_{s_1}$, it is clear that $a_{11}=0$. Then by
(\ref{F31}),(\ref{F33}), (\ref{F34}),(\ref{F32}) and Lemma \ref{P2.6}, we can get that
\begin{itemize}
\item $\mathcal{N}_{d_2}(f)\geq2$ implies that $a_{12}a_{21}=0$;
\item $\mathcal{N}_{d_3}(f)\geq2$ implies that $a_{13}a_{31}=0$;
\item $\mathcal{N}_{d_4}(f)\geq2$ implies that $a_{14}a_{41}=0$.
\end{itemize}
Then it follows that at least two of $a_{12}$, $a_{13}$ and $a_{14}$ are $0$ or at least two of
$a_{21}$, $a_{31}$ and $a_{41}$ are $0$. Without loss of generality, assume that $a_{12}=a_{13}=0$.
Again since $a_{11}=0$, by Lemma \ref{P2.6} we obtain $\pi_{p_{w(d_2d_3)}\comp g\comp i_{w(d_2d_3)}}(0)>d_2d_3$.
Thus $\mathcal{N}_{d_2d_3}(f)\geq2$. This indicates that the above claim holds.

The two claims imply that the necessity of the corollary is true. On the other hand, when $m=1$, let
$$f(x_{s_0+1},\cdots,x_{s_1})=(\lambda_1x_{s_0+1}+x_{s_0+2},\cdots,\lambda_1x_{s_1-1}+x_{s_1},
\lambda_1x_{s_1}+x_{s_0+1}^{r_1d_1+1}).$$
Then $\mathcal{N}_{d_1}(f)=r_1+1$ if $d_1=1$ and $\mathcal{N}_{d_1}(f)=r_1$ otherwise,
which indicates that $\Lambda$ is universal. When $m=2$, or when $m=3$ and
one of $d_1$, $d_2$ and $d_3$ is $1$, Example \ref{EX3.1} indicates that $\Lambda$ is universal.
Thus $\Lambda$ is universal if and only if $m\leq3$ and one of $d_1$, $d_2$ and $d_3$ is $1$ if $m=3$.
\end{proof}

\section{Monotonicity\label{Sec4}}
In this section we give the monotonicity of the numbers of periodic orbits of holomorphic germs
hidden at the origin with respect to their linearization matrices.
The monotonicity reveals the influence of the parts of linearization matrices to themselves
in terms of the numbers of periodic orbits hidden at the origin.

Let $JM_n$ denote all $n\times n$ complex Jordan matrices with each eigenvalue a root of unity and $JM=\cup_{n=1}^{\infty}JM_n$.
We define a partial order relation $\leq$ on $JM$ as follows. Let $\Lambda,\Lambda'\in JM$. Without loss of generality, we let $\Lambda$ be of the form (\ref{F1.1}), we say that $\Lambda'\leq\Lambda$ if
\begin{equation*}
\Lambda'={\rm diag}(A'_{k'_1},\cdots,A'_{k'_m}),
\end{equation*}
where $$A'_{k'_j}=\left(\begin{matrix}\lambda_j&1\\&\lambda_j&1\\&&\ddots&\ddots\\&&&\lambda_j&1\\&&&&\lambda_j
\end{matrix}
\right)_{k'_j\times k'_j},$$
$$1\leq j\leq m, k'_1+\cdots+k'_m=k',\ 0\leq k'_j\leq k_j\ {\rm and}\ 1\leq k'\leq k.$$
Let
$$\rho_q(\Lambda):=\inf\{\mathcal{N}_q(f): f\in\mathcal{O}_{\Lambda}\}$$
for any $\Lambda\in JM$ and any positive integer $q$. Then for the partial order relation given above,
we have the following theorem.
\begin{theorem}
Let $\Lambda, \Lambda'\in JM$ and $f\in\mathcal{R}_{\Lambda}$. If $\Lambda'\leq\Lambda$, then
\begin{enumerate}
\item there exists $\tilde f\in\mathcal{R}_{\Lambda'}$ such that $\mathcal{N}_q(f)\geq\mathcal{N}_q(\tilde f)$
for all $q\geq1$;
\item $\rho_q(\Lambda)\geq\rho_q(\Lambda')$ for all $q\geq 1$.
\end{enumerate}
\label{T2}
\end{theorem}
\begin{proof}
Let $\Lambda$ be of the form (\ref{F1.1}). By the definition of $\Lambda'\leq\Lambda$, we have that
\begin{equation*}
\Lambda'={\rm diag}(A'_{k'_1},\cdots,A'_{k'_m}),
\end{equation*}
where $$A'_{k'_j}=\left(\begin{matrix}\lambda_j&1\\&\lambda_j&1\\&&\ddots&\ddots\\&&&\lambda_j&1\\&&&&\lambda_j
\end{matrix}
\right)_{k'_j\times k'_j},$$
$$1\leq j\leq m,\ k'_1+\cdots+k'_m=k',\ 0\leq k'_j\leq k_j\ {\rm and}\ 1\leq k'\leq k.$$

We take a small perturbation $\Lambda''$ of $\Lambda$ such that
\begin{equation*}
\Lambda''={\rm diag}(A_{k_1}'',\cdots,A_{k_m}''),
\end{equation*}
where for every $1\leq j\leq m$,
$$A_{k_j}''=\left(\begin{matrix}\lambda^{(1)}_j&1\\&\lambda^{(2)}_j&1\\&&\ddots&\ddots\\&&&\lambda^{(k_j-1)}_j&1
\\&&&&\lambda^{(k_j)}_j\end{matrix}
\right)_{k_j\times k_j}$$
satisfies
\begin{itemize}
\item $\lambda^{(t)}_j=\lambda_j$ for $1\leq t\leq k'_j$;
\item $\lambda^{(k'_j+1)}_j,\cdots,\lambda^{(k_j)}_j$ are pairwise different and none of them is a root of unity;
\item The origin is an isolated fixed point of $g^{M(\Lambda)}$ for $g(\bold{x})=\tilde{\Lambda''}\bold{x}+\tau f(\bold{x})$;
\item $\mathcal{N}_{q}(f)\geq \mathcal{N}_{q}(g)$ for $q\in PE(\Lambda)$.(Thanks to the definition of
$\mathcal{N}_{q}(f)$)
\end{itemize}
By Lemma \ref{P8}, it is easy to see that
\begin{equation}
\label{F4.1}g\in\mathcal{O}_{\Lambda''}.
\end{equation}
Indeed, for any $l\geq1$, since the origin is an isolated fixed point of $g^{M(\Lambda)}$,
by Lemma \ref{P8} it is an isolated fixed point of $(g^{M(\Lambda)})^l$. Again
$(g^l)^{M(\Lambda)}=(g^{M(\Lambda)})^l$,
thus the origin is an isolated fixed point of $g^l$.

Let
\begin{equation*}
\Lambda'''={\rm diag}( A'''_{k_1},\cdots, A'''_{k_m})
\end{equation*}
with for every $1\leq j\leq m$,
$$A'''_{k_j}=\left(\begin{matrix}\lambda^{(1)}_j&1\\&\lambda^{(2)}_j&1
\\&&\ddots&\ddots\\&&&\lambda^{(k'_j-1)}_j&1\\&&&&\lambda^{(k'_j)}_j\\&&&&&\lambda^{(k'_j+1)}_j\\&&&&&&\ddots\\
&&&&&&&\lambda^{(k_j)}_j\end{matrix}\right)_{k_j\times k_j}.$$
It is easy to see that there exists an invertible $n\times n$ matrix $X$ such that
$\Lambda'''=X^{-1}\Lambda'' X.$
Let $h(\bold{x})=X\bold{x}$. Then by (\ref{F4.1}), we have $h^{-1}\comp g\comp h\in\mathcal{O}_{\Lambda'''}$.
By Theorem \ref{T3.1} and Lemma \ref{P4}, there exists $\tilde g\in\mathcal{R}_{\Lambda'''}$ such that
$\mathcal{N}_q(\tilde g)=\mathcal{N}_q(g)$ for all $q\geq1$.

Let $w(\Lambda')=(w_1w_2\cdots w_n)\in W_n^*$ such that
$w_j=1$ if and only if $j-s_{\theta(j)-1}\leq k_{\theta(j)}'$(Note that both $\theta(j)$ and $s_{\theta(j)-1}$ correspond to the matrix $\Lambda$
with the form (\ref{F1.1})). Let
$\tilde f=p_{w(\Lambda')}\comp\tilde g\comp i_{w(\Lambda')}$. Then
$\tilde f(\bold{x})=\tilde{\Lambda'}\bold{x}+p_{w(\Lambda')}\comp\tau\tilde g\comp i_{w(\Lambda')}(\bold{x})$.
By Theorem \ref{T3}, we have $\tilde f\in\mathcal{R}_{\Lambda'}$ and
$\mathcal{N}_{q}(\tilde g)=\mathcal{N}_{q}(\tilde f)$ for all $q\geq1$. Thus
$$\mathcal{N}_{q}(f)\geq \mathcal{N}_{q}(g)=\mathcal{N}_{q}(\tilde g)=\mathcal{N}_{q}(\tilde f)$$
for $q\in PE(\Lambda)$. Together with Theorem \ref{T1.2}, we can get that (1) holds.
At last, (2) follows easily from (1), and hence we complete the proof.
\end{proof}
\section{The proof of Theorem \ref{T1}\label{s5}}
To prove Theorem \ref{T1}, we first discuss five cases about $\Lambda$.
\subsection{Case $\uppercase\expandafter{\romannumeral1}$}
Let $\Lambda$ be an $n\times n$ matrix with the form (\ref{F1.1}) such that
$k_j=1$, $d_j>1$ and $d_j\nmid [d_1,\cdots,d_{j-1},d_{j+1},\cdots,d_m]$
for $j=1,2,\cdots,m$. Then it is clear that $m=n$. Furthermore, in this case, we have the following proposition.

\begin{proposition}
$\rho_{M(\Lambda)}(\Lambda)=1$ if and only if $d_1,\cdots,d_m$ are pairwise relatively prime.\label{p5.1}
\end{proposition}
We use the following lemma to prove Proposition \ref{p5.1}.
\begin{lemma}
\label{L1}Let $a_1,a_2,\cdots,a_n,r_1,r_2,\cdots,r_n$ be positive integers such that
 $1\leq r_j<a_j$ and $(r_j,a_j)=1$ for $j=1,2,\cdots,n$. Then there exists an integer $k_n$ such that
\begin{equation}
\label{F6}r_1^{(k_n)}r_2^{(k_n)}\cdots r_n^{(k_n)}\leq\frac{a_1a_2\cdots a_n}{[a_1,a_2,\cdots,a_n]},
\end{equation}
where $r_j^{(k_n)}$ is an integer and is determined uniquely by
$r_j^{(k_n)}= k_nr_j\ mod\ (a_j)$ and $1\leq r_j^{(k_n)}<a_j$ for $j=1,2,\cdots,n$. Moreover,
the above inequality {\rm(\ref{F6})} can be replaced by strict inequality if and only if
$a_1,a_2,\cdots,a_n$ are not pairwise relatively prime.
\end{lemma}
\begin{proof}

Let us prove the first part and the sufficiency of the second part in this lemma by mathematical induction.

When $n=1$, since $(r_1,a_1)=1$, there exist integers $l_1$ and $l_1'$ such that
\begin{equation}
\label{F5.1}l_1r_1+l_1'a_1=1.
\end{equation}
We take $k_1=l_1$. Then $r_1^{(k_1)}=1$. Thus $r_1^{(k_1)}=1\leq \frac{a_1}{[a_1]}$.

When $n=2$, since $(r_2,a_2)=1$, we have $(a_1r_2,a_2)=(a_1,a_2)$. Then there exist integers $l_2$ and $l_2'$
such that
\begin{equation}
a_1r_2l_2+l_2'a_2=(a_1,a_2).\label{F1}
\end{equation}
By (\ref{F5.1}), we see that $(l_1,a_1)=1$. Together with $(r_2,a_2)=1$, we have
\begin{equation}
\label{F5.2}(r_2l_1,(a_1,a_2))=1.
\end{equation}
If $(a_1,a_2)>1$, then, by (\ref{F5.2}), there exist integers $q$ and $r$ such that
\begin{equation}
\label{F5.3}r_2l_1=q(a_1,a_2)+r,
\end{equation}
where $1\leq r<(a_1,a_2)$. We take $k_2=l_1-l_2qa_1$. First, it is easy to see that $r_1^{(k_2)}=1$. Again,
by (\ref{F1}) and (\ref{F5.3}), it is easy to check $r_2^{(k_2)}=r$.
Thus
$$r_1^{(k_2)}r_2^{(k_2)}=r<(a_1,a_2)=\frac{a_1a_2}{[a_1,a_2]}.$$
If $(a_1,a_2)=1$, we take $k_2=l_1+(1-l_1r_2)l_2a_1$. By (\ref{F1}), we have $r_2^{(k_2)}=1$.
Again, it is easy to see that $r_1^{(k_2)}=1$. Thus
$$r_1^{(k_2)}r_2^{(k_2)}=1=\frac{a_1a_2}{[a_1,a_2]}.$$

We suppose that when $n=m\geq 2$, this conclusion holds. Now we need to consider the case that $n=m+1$.
If $a_1,a_2,\cdots,a_{m+1}$ are not pairwise relatively prime, we may change the order of $a_1,a_2,\cdots,a_{m+1}$,
if necessary,
such that $a_1,\cdots,a_m$ are not pairwise relatively prime. Since $(r_{m+1},a_{m+1})=1$,
we have $([a_1,\cdots,a_m]r_{m+1},a_{m+1})=([a_1,\cdots,a_m],a_{m+1})$.
Then there exist integers $l_{m+1}$ and $l_{m+1}'$ such that
\begin{equation}
l_{m+1}[a_1,\cdots,a_m]r_{m+1}+l_{m+1}'a_{m+1}=([a_1,\cdots,a_m],a_{m+1}).\label{F2}
\end{equation}
Dividing $r_{m+1}k_m$ by $([a_1,\cdots,a_m],a_{m+1})$, we get that there exist integers $q'$ and $r'$ such that
\begin{equation}
r_{m+1}k_m=q'([a_1,\cdots,a_m],a_{m+1})+r',\label{F3}
\end{equation}
here $1\leq r'<([a_1,\cdots,a_m],a_{m+1})$ if $([a_1,\cdots,a_m],a_{m+1})\nmid r_{m+1}k_m$
and $r'=([a_1,\cdots,a_m],a_{m+1})$  otherwise. We take
\begin{equation}
\label{F5.20}k_{m+1}=k_m-l_{m+1}q'[a_1,\cdots,a_m].
\end{equation}
Then by (\ref{F2}) and (\ref{F3}), we have $r_{m+1}^{(k_{m+1})}=r'$ with
$1\leq r_{m+1}^{(k_{m+1})}\leq a_{m+1}$.
By (\ref{F5.20}), it is easy to see that $r_1^{(k_{m+1})}\cdots r_m^{(k_{m+1})}=r_1^{(k_m)}\cdots r_m^{(k_m)}$.
Thus $$r_1^{(k_{m+1})}\cdots r_{m+1}^{(k_{m+1})}\leq \frac{a_1\cdots a_m}{[a_1,\cdots,a_m]}r'\leq
\frac{a_1\cdots a_m}{[a_1,\cdots,a_m]}([a_1,\cdots,a_m],a_{m+1})=\frac{a_1\cdots a_{m+1}}{[a_1,\cdots,a_{m+1}]}$$
and if $a_1,a_2,\cdots,a_{m+1}$ are not pairwise relatively prime,
the above first inequality can be replaced by strict inequality.

Next we prove that we can make an appropriate choice of $k_m$ to ensure $r_{m+1}^{(k_{m+1})}<a_{m+1}$.
Indeed, if $([a_1,\cdots,a_m],a_{m+1})<a_{m+1}$, it is clear. Assume $([a_1,\cdots,a_m],a_{m+1})=a_{m+1}$. Then
$$2\leq a_{m+1}\mid[a_1,\cdots,a_m],$$
and hence there exists $j\in\{1,2,\cdots,m\}$ such that
\begin{equation}
\label{F5.0}(a_j,a_{m+1})\not=1.
\end{equation}
According to the case $n=1$, the case $n=2$ and (\ref{F5.20}), there exists $k_m$ such that
$(k_m,a_1)=1$. By replacing $a_1$ by $a_j$ in the above progress, we have that there exists $k_m$ such that $(k_m,a_j)=1$. In this case, by (\ref{F5.0}) we have $a_{m+1}\nmid k_m$. Again $(a_{m+1},r_{m+1})=1$,
thus we have
$$a_{m+1}\nmid r_{m+1}k_m,$$
that is,
$$([a_1,\cdots,a_m],a_{m+1})\nmid r_{m+1}k_m.$$
Then
$$r_{m+1}^{(k_{m+1})}=r'<([a_1,\cdots,a_m],a_{m+1})=a_{m+1}.$$
Thus this conclusion holds for $n=m+1$, and this completes mathematical induction.

At last, if $a_1,\cdots,a_n$ are pairwise relatively prime, then $\frac{a_1\cdots a_n}{[a_1,\cdots,a_n]}=1$,
and again together with the first part in this lemma, we get that the necessity of the second part
in this lemma holds.
\end{proof}

\begin{proof}[\textbf{PROOF OF PROPOSITION \protect\ref{p5.1}}]
Let $f\in\mathcal{R}_{\Lambda}$ and $\tau f=(f_1,f_2,\cdots,f_n)$. Since
$$d_j\nmid [d_1,\cdots,d_{j-1},d_{j+1},\cdots,d_n]$$
for $j=1,2,\cdots,n$, we have $x_j\mid f_j$ for $j=1,2,\cdots,n$.
Let
$$g=(\frac{f_1}{x_1},\frac{f_2}{x_2},\cdots,\frac{f_n}{x_n}).$$
By Proposition $\ref{L2}$, we have that
\begin{equation}
\label{F5.22}\mathcal{N}_{M(\Lambda)}(f)=\frac{\pi_g(0)}{M(\Lambda)}.
\end{equation}
Let $\lambda_j=e^{2\pi i\frac{r_j}{d_j}}$, where $1\leq r_j<d_j$ and $(r_j,d_j)=1$ for $j=1,2,\cdots,n$.

We first consider the case that $d_1,\cdots,d_n$ are not pairwise relatively prime. Then by Lemma \ref{L1},
there exists an integer $k_n$ such that
$$r_1^{(k_n)}r_2^{(k_n)}\cdots r_n^{(k_n)}<\frac{d_1d_2\cdots d_n}{[d_1,d_2,\cdots,d_n]},$$
where $r_j^{(k_n)}$ is an integer and determined uniquely by
$r_j^{(k_n)}= k_nr_j\ mod\ (d_j)$ and $1\leq r_j^{(k_n)}<d_j$ for $j=1,2,\cdots,n$.
According to the definition of $g$, it is easy to see that $g\comp\tilde\Lambda=g$.
Then
\begin{equation}
\label{F5.4}g\comp\tilde\Lambda^{k_n}=g.
\end{equation}
Combining (\ref{F5.4}) and Lemma \ref{P2.6}, we have that the multiplicity of
$g(y_1^{r_1^{(k_n)}\frac{M(\Lambda)}{d_1}},\cdots,y_n^{r_n^{(k_n)}\frac{M(\Lambda)}{d_n}})$ at the origin
is not less than $M(\Lambda)^n$. Then $$\pi_g(0)\geq\frac{M(\Lambda)^n}{{r_1^{(k_n)}\frac{M(\Lambda)}{d_1}}
\cdots r_n^{(k_n)}\frac{M(\Lambda)}{d_n}}>[d_1,\cdots,d_n]=M(\Lambda),$$
and hence, by (\ref{F5.22}), $\mathcal{N}_{M(\Lambda)}(f)>1.$

Next, we consider the other case, that is, $d_1,\cdots,d_n$ are pairwise relatively prime.
Let $$g(x_1,\cdots,x_n)=(x_1^{d_1},\cdots,x_n^{d_n}).$$
Then $\mathcal{N}_{M(\Lambda)}(f)=\frac{\pi_g(0)}{M(\Lambda)}=1$.

At last, combining these two cases and the arbitrariness of $f$, the proof of this proposition is completed.
\end{proof}
\subsection{Case $\uppercase\expandafter{\romannumeral2}$}
Let $\Lambda$ be an $n\times n$ matrix with the form (\ref{F1.1}) such that $d_j=d\geq 1$ for $j=1,2,\cdots,m$.
Then, in this case, we have the following proposition.

\begin{proposition}
The matrix $\Lambda$ is universal if and only if $m=1$.\label{P5.2}
\end{proposition}
\begin{proof}
The sufficiency follows from Corollary \ref{Cr3.2}. When $d=1$, it is easy to check that the proposition holds. Thus
we only need to prove that $\Lambda$ is not universal when $m\geq 2$ and $d\geq 2$.

Let $\lambda_j=e^{2\pi i\frac{r_j}{d}}$ such that $1\leq r_j<d$ and $(r_j,d)=1$ for $j=1,\cdots,m$.
Then there exists an integer $k$ and a permutation $j_1\cdots j_m$ of $12\cdots m$ such that
$1=r_{j_1}^{(k)}\leq\cdots\leq r_{j_m}^{(k)}<d$,
where $r_{j_s}^{(k)}$ is an integer and
is determined uniquely by $r_{j_s}^{(k)}=kr_{j_s}\ mod\ (d)$
and $1\leq r_{j_s}^{(k)}<d$ ($s=1,2,\cdots,m$).

We claim that $\rho_{M(\Lambda)}(\Lambda)=1$ implies
$r_{j_1}^{(k)}\mid_{\not=} r_{j_2}^{(k)}\mid_{\not=}\cdots\mid_{\not=}r_{j_m}^{(k)}\mid d+1$.

Indeed, according to Theorem \ref{T2}, we can suppose that $k_1=k_2=\cdots=k_m=1$, that is, $m=n$.
Let $f\in\mathcal{R}_{\Lambda}$ such that $\mathcal{N}_{M(\Lambda)}(f)=1$.
By (\ref{F3.7}), we have that $\pi_{\tau f}(0)=d+1$. Together with Lemma \ref{P4}, we have that
\begin{equation}
\label{F5.5}\pi_{\tau f(x_1^{r_1^{(k)}},\cdots,x_n^{r_n^{(k)}})}(0)=(d+1)r_2^{(k)}\cdots r_n^{(k)}.
\end{equation}
Since $f\in\mathcal{R}_{\Lambda}$, we have $f\comp\tilde\Lambda=\tilde\Lambda f$.
Then $\tau f\comp\tilde\Lambda=\tilde\Lambda(\tau f)$ and hence $\tau f\comp\tilde\Lambda^k=\tilde\Lambda^k(\tau f)$.
Again $\tau f$ has a vanishing linear part,
thus we have that $$\tau f(x_1^{r_1^{(k)}},\cdots,x_n^{r_n^{(k)}})=(f_1,\cdots,f_n),$$
where $$f_{j_1}(x_1,\cdots,x_n)=
\sum_{r_1^{(k)}i_1+\cdots+r_n^{(k)}i_n=d+1}c_{i_1\cdots i_n}x_1^{r_1^{(k)}i_1}
\cdots x_n^{r_n^{(k)}i_n}+{\rm higher\ terms}\ (c_{i_1\cdots i_n}\in\mathbb{C})$$
and
$$f_{j_s}(x_1,\cdots,x_n)=\sum_{r_{j_1}^{(k)}i_1+\cdots+r_{j_{s-1}}^{(k)}i_{s-1}=
r_{j_s}^{(k)}}c_{si_1\cdots i_{s-1}}x_{j_1}^{r_{j_1}^{(k)}i_1}\cdots x_{j_{s-1}}^{r_{j_{s-1}}^{(k)}i_{s-1}}+
{\rm higher\ terms}\ (c_{si_1\cdots i_{s-1}}\in\mathbb{C})$$ for $s=2,\cdots,n$.
Let $F=(F_1,\cdots,F_n)$ with
$$F_{j_1}(x_1,\cdots,x_n)=\sum_{r_1^{(k)}i_1+\cdots+r_n^{(k)}i_n=d+1}c_{i_1\cdots i_n}x_1^{r_1^{(k)}i_1}\cdots
x_n^{r_n^{(k)}i_n}$$
and
$$F_{j_s}(x_1,\cdots,x_n)=\sum_{r_{j_1}^{(k)}i_1+\cdots+r_{j_{s-1}}^{(k)}i_{s-1}=r_{j_s}^{(k)}}
c_{si_1\cdots i_{s-1}}x_{j_1}^{r_{j_1}^{(k)}i_1}\cdots x_{j_{s-1}}^{r_{j_{s-1}}^{(k)}i_{s-1}}$$ for $s=2,\cdots,n$.
By (\ref{F5.5}) and Lemma \ref{P2.6}, the origin is an isolated zero of $F$.
Then by Lemma \ref{PR1} we can get that
$F_{j_s}$ has a nonzero monomial proportional to $x_{j_{s-1}}^{r_{j_s}^{(k)}}$ for $s=2,\cdots,n$ and
$F_{j_1}$ has a nonzero monomial proportional to $x_{j_n}^{d+1}$.
Again, since $\tau f$ has a vanishing linear part, we have that $r_{j_1}^{(k)}\mid_{\not=}
r_{j_2}^{(k)}\mid_{\not=}\cdots\mid_{\not=}r_{j_m}^{(k)}\mid d+1$. Thus we complete the proof of the claim.

According to the claim, it is sufficient for completing the proof of this proposition to show that
if $r_{j_1}^{(k)}\mid_{\not=} r_{j_2}^{(k)}\mid_{\not=}\cdots\mid_{\not=}r_{j_m}^{(k)}\mid d+1$,
then $\mathcal{N}_{M(\Lambda)}(f)\not=2$ for all $f\in\mathcal{R}_{\Lambda}$. To prove this,
we suppose that $r_{j_1}^{(k)}\mid_{\not=} r_{j_2}^{(k)}\mid_{\not=}\cdots\mid_{\not=}r_{j_m}^{(k)}\mid d+1$
and there exists $g\in\mathcal{R}_{\Lambda}$ such that $\mathcal{N}_{M(\Lambda)}(g)=2$.

Let $\tau g=(g_1,\cdots,g_n).$ We first consider the case that $k_1=k_2=\cdots=k_m=1$, that is, $m=n$.
Similarly in the proof of the above claim, we can get that
\begin{equation}
\label{F5.23}\pi_{\tau g(x_1^{\alpha_1^{(1)}},\cdots,x_n^{\alpha_n^{(1)}})}(0)=(2d+1)\alpha_{j_2}^{(1)}\cdots
\alpha_{j_n}^{(1)},
\end{equation}
$$g_{j_s}(x_1^{\alpha_1^{(1)}},\cdots,x_n^{\alpha_n^{(1)}})=\sum_{\alpha_{j_1}^{(1)}i_1+\cdots+
\alpha_{j_{s-1}}^{(1)}i_{s-1}
=\alpha_{j_s}^{(1)}}c_{si_1\cdots i_{s-1}}^{(1)}x_{j_1}^{\alpha_{j_1}^{(1)}i_1}\cdots
x_{j_{s-1}}^{\alpha_{j_{s-1}}^{(1)}i_{s-1}}+{\rm higher\ terms},$$
for $s=2,\cdots,n$ and
$$g_{j_1}(x_1^{\alpha_1^{(1)}},\cdots,x_n^{\alpha_n^{(1)}})=\sum_{\alpha_1^{(1)}i_1+\cdots+\alpha_n^{(1)}i_n=d+1}
c_{i_1\cdots i_n}^{(1)}x_1^{\alpha_1^{(1)}i_1}\cdots x_n^{\alpha_n^{(1)}i_n}+{\rm higher\ terms},$$
where $c_{i_1\cdots i_n}^{(1)}$ and $c_{si_1\cdots i_{s-1}}^{(1)}$ are complex numbers and
$\alpha_j^{(1)}=r_j^{(k)}$ for $s=2,\cdots,n$ and $j=1,2,\cdots,n$.
Together with Lemma \ref{P2.6}, this implies
\begin{equation}
\label{F5.6}g_{j_2}(x_1^{\alpha_1^{(1)}},\cdots,x_n^{\alpha_n^{(1)}})=c^{(1)}x_{j_1}^{\alpha_{j_2}^{(1)}}
+{\rm higher\ terms}
\end{equation}
with $c^{(1)}\not=0$. Indeed, if not, that is, $c^{(1)}=0$, then the least degree of nonzero monomials of
$g_{j_2}$ is $\alpha_{j_2}^{(1)}+d$, and hence
$$\pi_{\tau g(x_1^{\alpha_1^{(1)}},\cdots,x_n^{\alpha_n^{(1)}})}(0)\geq
(d+1)(\alpha_{j_2}^{(1)}+d)\alpha_{j_3}^{(1)}\cdots \alpha_{j_n}^{(1)}.$$
This contradicts (\ref{F5.23}).

Let $k_{(1)}=k\frac{d+1}{r_{j_n}^{(k)}}$. Then $1=r_{j_n}^{(k_{(1)})}\mid_{\not=}r_{j_1}^{(k_{(1)})}
\mid_{\not=} r_{j_2}^{(k_{(1)})}\mid_{\not=}\cdots\mid_{\not=}r_{j_{n-1}}^{(k_{(1)})}\mid d+1$.
Together with (\ref{F5.6}), we have $$g_{j_2}(x_1^{\alpha_1^{(2)}},\cdots,x_n^{\alpha_n^{(2)}})
=c^{(1)}x_{j_1}^{\alpha_{j_2}^{(2)}}+\sum_{\substack{i_1+i_2=
\alpha_{j_2}^{(2)};\\i_2\geq1}}c_{i_1i_2}^{(2)}x_{j_1}^{i_1}x_{j_n}^{i_2}+{\rm higher\ terms},$$ and similarly
$$g_{j_1}(x_1^{\alpha_1^{(2)}},\cdots,x_n^{\alpha_n^{(2)}})=c^{(2)}x_{j_n}^{\alpha_{j_1}^{(2)}}+{\rm higher\ terms},$$
where $c^{(2)}\not=0$ and $\alpha_j^{(2)}=r_j^{(k_{(1)})}$ for $j=1,2,\cdots,n$.

Continuing in this manner, and at last let $k_{(n-1)}=k\frac{d+1}{r_{j_2}^{(k)}}$. Then
$1=r_{j_2}^{(k_{(n-1)})}\mid_{\not=}r_{j_3}^{(k_{(n-1)})}\mid_{\not=}\cdots\mid_{\not=}r_{j_n}^{(k_{(n-1)})}
\mid_{\not=}r_{j_1}^{(k_{(n-1)})}\mid d+1$. Similarly, we have
\begin{align*}
g_{j_1}(x_1^{\alpha_1^{(n)}},\cdots,x_n^{\alpha_n^{(n)}})&=c^{(2)}x_{j_n}^{\alpha_{j_1}^{(n)}}+
\sum_{\substack{i_2+\cdots+i_n=
\alpha_{j_1}^{(n)};\\i_n<\alpha_{j_1}^{(n)}}}c_{i_2i_3\cdots i_n}^{(n)}x_{j_2}^{i_2}\cdots
x_{j_n}^{i_n}+{\rm higher\ terms};\\
g_{j_n}(x_1^{\alpha_1^{(n)}},\cdots,x_n^{\alpha_n^{(n)}})&=c^{(3)}x_{j_{n-1}}^{\alpha_{j_n}^{(n)}}+
\sum_{\substack{i_2+\cdots+i_{n-1}=
\alpha_{j_n}^{(n)};\\i_{n-1}<\alpha_{j_n}^{(n)}}}c_{i_2i_3\cdots i_{n-1}}^{(n)}x_{j_2}^{i_2}
\cdots x_{j_{n-1}}^{i_{n-1}}+{\rm higher\ terms};\\
&\vdots\\
g_{j_3}(x_1^{\alpha_1^{(n)}},\cdots,x_n^{\alpha_n^{(n)}})&=c^{(n)}x_{j_2}^{\alpha_{j_3}^{(n)}}+{\rm higher\ terms};\\
g_{j_2}(x_1^{\alpha_1^{(n)}},\cdots,x_n^{\alpha_n^{(n)}})&=c^{(1)}x_{j_1}^{\alpha_{j_2}^{(n)}+d}+
\sum_{\substack{i_1+\cdots+i_n=
\alpha_{j_2}^{(n)}+d;\\i_2+\cdots+i_n\geq1}}c_{i_1i_2\cdots i_n}^{(n)}x_{j_1}^{i_1}\cdots x_{j_n}^{i_n}+
{\rm higher\ terms},
\end{align*}
where $c^{(n)}\not=0$ and $\alpha_j^{(n)}=r_j^{(k_{(n-1)})}$ for $j=1,\cdots,n$. Let
$$G(x_1,\cdots,x_n)=(G_1,G_2,\cdots,G_n)$$
with
\begin{align*}
G_{j_1}(x_1,\cdots,x_n)&=c^{(2)}x_{j_n}^{\alpha_{j_1}^{(n)}}+\sum_{\substack{i_2+\cdots+i_n=
\alpha_{j_1}^{(n)};\\i_n<\alpha_{j_1}^{(n)}}}c_{i_2i_3\cdots i_n}^{(n)}x_{j_2}^{i_2}\cdots x_{j_n}^{i_n};\\
G_{j_n}(x_1,\cdots,x_n)&=c^{(3)}x_{j_{n-1}}^{\alpha_{j_n}^{(n)}}+\sum_{\substack{i_2+\cdots+i_{n-1}=
\alpha_{j_n}^{(n)};\\i_{n-1}<\alpha_{j_n}^{(n)}}}c_{i_2i_3\cdots i_{n-1}}^{(n)}x_{j_2}^{i_2}\cdots
x_{j_{n-1}}^{i_{n-1}};\\
&\vdots\\
G_{j_3}(x_1,\cdots,x_n)&=c^{(n)}x_{j_2}^{\alpha_{j_3}^{(n)}};\\
G_{j_2}(x_1,\cdots,x_n)&=c^{(1)}x_{j_1}^{1+d}+\sum_{\substack{i_1+\cdots+i_n=
1+d;\\i_2+\cdots+i_n\geq1}}c_{i_1i_2\cdots i_n}^{(n)}x_{j_1}^{i_1}\cdots x_{j_n}^{i_n}.
\end{align*}
Then by Lemma \ref{P1} and Lemma \ref{PR1}, we can get that the origin is an isolated zero of $G$ and
$\pi_G(0)=\alpha_{j_1}^{(n)}\alpha_{j_n}^{(n)}\cdots\alpha_{j_3}^{(n)}(d+1)$. Together with Lemma \ref{P2.6},
we have $\pi_{\tau g}(0)=\frac{1}{\alpha_{j_1}^{(n)}\alpha_{j_n}^{(n)}\cdots\alpha_{j_3}^{(n)}}\pi_G(0)=d+1$ and
hence $\mathcal{N}_{M(\Lambda)}(g)=1$.

For general case, by Proposition \ref{p3.3}, there exists $\tilde g\in\mathcal{R}_{\Lambda}$ with
$\tau\tilde g=(\tilde g_1,\tilde g_2,\cdots,\tilde g_n)$ such that $\mathcal{N}_{M(\Lambda)}(\tilde g)=
\mathcal{N}_{M(\Lambda)}(g)=2$ and
$$\tilde g_j=g_j=x_{j+1}+{\rm higher\ terms}$$
for $j\in\{1,2,\cdots,n\}\setminus\{s_1,\cdots,s_m\}$. Moreover,
all variables of every monomial of $\tilde g_{s_1},\cdots,\tilde g_{s_m}$ come from
$x_{s_0+1},x_{s_1+1},\cdots,x_{s_{m-1}+1}$.

Next, we consider
$$\tilde g_j(x_1^{\alpha_1^{(1)}},\cdots,x_{s_1-1},x_{s_1},x_{s_1+1}^{\alpha_2^{(1)}},\cdots,x_{s_2-1},x_{s_2},
\cdots,x_{s_{m-1}+1}^{\alpha_m^{(1)}},\cdots,x_{s_m-1},x_{s_m}), j=1,\cdots,n$$
by the similar way as the above special case, and
we can get that
\begin{align*}
&\tilde g_{s_{j_1}}(x_1^{\alpha_1^{(m)}},\cdots,x_{s_1-1},x_{s_1},x_{s_1+1}^{\alpha_2^{(m)}},\cdots,x_{s_2-1},x_{s_2},
\cdots,x_{s_{m-1}+1}^{\alpha_m^{(m)}},\cdots,x_{s_m-1},x_{s_m})\\
&=c^{(2)}x_{s_{j_m-1}+1}^{\alpha_{j_1}^{(m)}}+\sum_{\substack{i_2+\cdots+i_m=
\alpha_{j_1}^{(m)};\\i_m<\alpha_{j_1}^{(m)}}}c_{i_2i_3\cdots i_m}^{(m)}x_{s_{j_2-1}+1}^{i_2}
\cdots x_{s_{j_m-1}+1}^{i_m}+{\rm higher\ terms},\\
&\tilde g_{s_{j_m}}(x_1^{\alpha_1^{(m)}},\cdots,x_{s_1-1},x_{s_1},x_{s_1+1}^{\alpha_2^{(m)}},\cdots,x_{s_2-1},x_{s_2},
\cdots,x_{s_{m-1}+1}^{\alpha_m^{(m)}},\cdots,x_{s_m-1},x_{s_m})\\
&=c^{(3)}x_{s_{j_{m-1}-1}+1}^{\alpha_{j_m}^{(m)}}+\sum_{\substack{i_2+\cdots+i_{m-1}=
\alpha_{j_m}^{(m)};\\i_{m-1}<\alpha_{j_m}^{(m)}}}c_{i_2i_3\cdots i_{m-1}}^{(m)}x_{s_{j_2-1}+1}^{i_2}
\cdots x_{s_{j_{m-1}-1}+1}^{i_{m-1}}+{\rm higher\ terms},\\
\end{align*}
$$\vdots$$
\begin{align*}
&\tilde g_{s_{j_3}}(x_1^{\alpha_1^{(m)}},\cdots,x_{s_1-1},x_{s_1},x_{s_1+1}^{\alpha_2^{(m)}},\cdots,x_{s_2-1},x_{s_2},
\cdots,x_{s_{m-1}+1}^{\alpha_m^{(m)}},\cdots,x_{s_m-1},x_{s_m})\\
&=c^{(m)}x_{s_{j_2-1}+1}^{\alpha_{j_3}^{(m)}}+{\rm higher\ terms},\\
&\tilde g_{s_{j_2}}(x_1^{\alpha_1^{(m)}},\cdots,x_{s_1-1},x_{s_1},x_{s_1+1}^{\alpha_2^{(m)}},\cdots,x_{s_2-1},x_{s_2},
\cdots,x_{s_{m-1}+1}^{\alpha_m^{(m)}},\cdots,x_{s_m-1},x_{s_m})\\
&=c^{(1)}x_{s_{j_1-1}+1}^{\alpha_{j_2}^{(m)}+d}+\sum_{\substack{i_1+\cdots+i_m=
\alpha_{j_2}^{(m)}+d;\\i_2+\cdots+i_m\geq1}}c_{i_1i_2\cdots i_m}^{(m)}x_{s_{j_1-1}+1}^{i_1}\cdots
x_{s_{j_m-1}+1}^{i_m}+{\rm higher\ terms}\\
\end{align*}
and
\begin{align*}
&\tilde g_j(x_1^{\alpha_1^{(m)}},\cdots,x_{s_1-1},x_{s_1},x_{s_1+1}^{\alpha_2^{(m)}},\cdots,x_{s_2-1},x_{s_2},
\cdots,x_{s_{m-1}+1}^{\alpha_m^{(m)}},\cdots,x_{s_m-1},x_{s_m})\\
&=x_{j+1}+{\rm higher\ terms}\\
\end{align*}
for $j\in\{1,2,\cdots,n\}\setminus\{s_1,\cdots,s_m\}.$
Similarly, by Lemma \ref{P1}, Lemma \ref{PR1} and Lemma \ref{P2.6},
we can get that $\pi_{\tau\tilde g}(0)=d+1$ and
hence $\mathcal{N}_{M(\Lambda)}(g)=\mathcal{N}_{M(\Lambda)}(\tilde g)=1.$
This contradicts the above supposition. Thus we complete the proof.
\end{proof}
\subsection{Case $\uppercase\expandafter{\romannumeral3}$\label{s5.3}}
Let $\Lambda$ be an $n\times n$ matrix with the form (\ref{F1.1}) such that $1\leq d_1\mid d_2\mid\cdots\mid d_m$ and
there exist $s$ and $t\in \{1,2,\cdots,m\}$ such that $d_s\not=d_t$. Then, in this case, we have the following proposition.

\begin{proposition}
\label{P5.3}$d_j\mid_{\not=}d_{j+1}$ and $\lambda_j=\lambda_{j+1}^{\frac{d_{j+1}}{d_j}}$ for $j=1,2,\cdots,m-1$
if and only if there exists
$h\in\mathcal{R}_{\Lambda}$ such that
\begin{equation}
\label{F5.13}\mathcal{N}_q(h)=\left\{\begin{matrix}1\ ,&q\in PE(\Lambda)\ and\ q\geq2\\2\ ,
&1\in PE(\Lambda)\ and\ q=1\\1\ ,&1\not\in PE(\Lambda)\ and\ q=1\\0\ ,&else\end{matrix}\right..
\end{equation}
\end{proposition}
\begin{proof}
We first prove the sufficiency. Assume that there exists $h\in\mathcal{R}_{\Lambda}$ such that (\ref{F5.13}) holds.
Without loss of generality, let $s<t$. First, let us prove
\begin{equation}
\label{F5.14}\lambda_s=\lambda_t^{\frac{d_t}{d_s}}.
\end{equation}
Let $\Lambda_{st}={\rm diag}\{\lambda_s, \lambda_t\}$. By Theorem \ref{T1.2} and Theorem \ref{T2}, we have
$1\leq\rho_{M(\Lambda_{st})}(\Lambda_{st})\leq\rho_{M(\Lambda_{st})}(\Lambda)$. By (\ref{F5.13}), we have
$\rho_{M(\Lambda_{st})}(\Lambda)=1$. Thus $\rho_{M(\Lambda_{st})}(\Lambda_{st})=1$.
Let $g\in\mathcal{R}_{\Lambda_{st}}$ such that $\mathcal{N}_{M(\Lambda_{st})}(g)=1$
and $\tau g=(g_s,g_t)$. It is easy to see that $\tau g\comp\Lambda_{st}=\Lambda_{st}(\tau g)$.
This implies that $x_t$ is a factor of $g_t$. Let $f=(f_s,f_t)$ with $f_s=g_s$ and $f_t=\frac{g_t}{x_t}$. Then
\begin{equation}
\label{F5.7}f_s\comp\Lambda_{st}=\lambda_s f_s,
\end{equation}
\begin{equation}
\label{F5.8}f_t\comp\Lambda_{st}=f_t.
\end{equation}
And by Proposition \ref{L2}, we have
\begin{equation}
\label{F5.9}\pi_f(0)=M(\Lambda_{st})\mathcal{N}_{M(\Lambda_{st})}(g)=d_t.
\end{equation}
Let $\lambda_s=e^{2\pi i\frac{r_s}{d_s}}$ and $\lambda_t=e^{2\pi i\frac{r_t}{d_t}}$,
where $(r_s,d_s)=(r_t,d_t)=1$, $1\leq r_s\leq d_s$ and $1\leq r_t<d_t$.
Then there exist integers $k,p,q$ and a positive integer $1\leq r\leq d_s$ such that
\begin{equation}
\label{F5.10}kr_t+qd_t=1
\end{equation}
and
\begin{equation}
\label{F5.11}kr_s=pd_s+r.
\end{equation}
By (\ref{F5.7}) and (\ref{F5.8}), we have
\begin{equation}
\label{F5.24}f_s\comp\Lambda_{st}^k=\lambda_s^k f_s
\end{equation}
and
$$f_t\comp\Lambda_{st}^k=f_t.$$
Thus
\begin{equation}
\label{F5.12}f_s(x_s^{r\frac{d_t}{d_s}},x_t)=c_1x_t^{r\frac{d_t}{d_s}}+{\rm higher\ terms}
\end{equation}
and
$$f_t(x_s^{r\frac{d_t}{d_s}},x_t)=\sum_{i_1+i_2=d_t}c_{i_1i_2}x_s^{i_1}x_t^{i_2}+{\rm higher\ terms},$$
where $c_1,c_{i_1i_2}$ are complex numbers. By (\ref{F5.9}) and Lemma \ref{P4},
$$\pi_{f(x_s^{r\frac{d_t}{d_s}},x_t)}(0)=d_tr\frac{d_t}{d_s}.$$
Together with Lemma \ref{P2.6}, we have $$f_t(x_s^{r\frac{d_t}{d_s}},x_t)=c_{d_t0}x_s^{d_t}+{\rm higher\ terms}$$
and $c_1c_{d_t0}\not=0$. This implies $r\frac{d_t}{d_s}\mid d_t$ and hence $r\mid d_s$.
Together with $(r_s,d_s)=1$ and (\ref{F5.10}), we have $(r,kr_s)=1$. Again, by (\ref{F5.11}), $r\mid kr_s$.
Then we have $r=1$. So (\ref{F5.12}) can be expressed as
$$f_s(x_s^{\frac{d_t}{d_s}},x_t)=c_1x_t^{\frac{d_t}{d_s}}+{\rm higher\ terms}.$$
Then $c_1x_t^{\frac{d_t}{d_s}}$ is a nonzero monomial of $f_s(x_s,x_t)$.
Together with (\ref{F5.24}), we get $\lambda_s^k=\lambda_t^{\frac{kd_t}{d_s}}$.
Since $(k,d_s)=1$, there exist integers $p'$ and $q'$ such that $kp'+d_sq'=1$. Then
$$\lambda_s=\lambda_s^{kp'+d_sq'}=(\lambda_s^k)^{p'}=\lambda_t^{\frac{p'kd_t}{d_s}}
=\lambda_t^{\frac{p'kd_t}{d_s}}\lambda_t^{\frac{q'd_sd_t}{d_s}}=\lambda_t^{\frac{d_t}{d_s}},$$
and hence (\ref{F5.14}) holds.

Combining (\ref{F5.14}) and the claim in the proof of Proposition \ref{P5.2},
we can get that there exist an integer $1<j\leq m$ such that
$d_1\mid_{\not=}d_2\mid_{\not=}\cdots\mid_{\not=}d_j=d_{j+1}=\cdots=d_m$ and $\lambda_v=\lambda_j^{\frac{d_j}{d_v}}$
for $v=1,2,\cdots,j$.

Indeed, it is clear that there exists an integer $1<j\leq m$ such that
$d_1\mid d_2\mid\cdots\mid d_{j-1}\mid_{\not=}d_j=d_{j+1}=\cdots=d_m$.
By (\ref{F5.14}), we have
\begin{equation}
\label{F5.15}\lambda_v=\lambda_j^{\frac{d_j}{d_v}}
\end{equation}
for $v=1,2,\cdots,j-1$.
If there exists an integer $v_0\in\{1,2,\cdots,j-2\}$ such that $d_{v_0}=d_{v_0+1}$, then by (\ref{F5.15}),
we have $\lambda_{v_0}=\lambda_{v_0+1}$. Let $\Lambda_{v_0(v_0+1)}={\rm diag}\{\lambda_{v_0},\lambda_{v_0+1}\}$.
Then by Theorem \ref{T2} and the claim in the proof of Proposition \ref{P5.2},
we have
$$\rho_{d_{v_0}}(\Lambda)\geq\rho_{d_{v_0}}(\Lambda_{v_0(v_0+1)})\left\{\begin{matrix}\geq2\ ,
&d_{v_0}\geq2\\ \geq4\ ,&d_{v_0}=1\end{matrix}\right..$$
This contradicts (\ref{F5.13}) and hence $d_1\mid_{\not=}d_2\mid_{\not=}\cdots\mid_{\not=}d_j$.

At last, to complete the proof of the sufficiency, we only need to prove $j=m$.
Indeed, if $j<m$, then we consider $d_{j-1},d_j$ and $d_{j+1}$.
Let $\lambda_{j-1}=e^{2\pi i\frac{r_{j-1}}{d_{j-1}}}$, $\lambda_j=e^{2\pi i\frac{r_j}{d_j}}$ and
$\lambda_{j+1}=e^{2\pi i\frac{r_{j+1}}{d_{j+1}}}$ such that
$(r_{j-1},d_{j-1})=(r_j,d_j)=(r_{j+1},d_{j+1})=1$, $1\leq r_{j-1}\leq d_{j-1}$, $1\leq r_j<d_j$ and
$1\leq r_{j+1}<d_{j+1}$. Since $(r_j,d_j)=1$, $(r_{j+1},d_{j+1})=1$ and $d_j=d_{j+1}>1$,
there exist integers $l, p_1, p_2$ and $r_{j+1}^{(l)}$ such that $lr_j+p_1d_j=1$, $lr_{j+1}=p_2d_{j+1}+r_{j+1}^{(l)}$
and $1\leq r_{j+1}^{(l)}<d_{j+1}$. Then 
$\lambda_j^l=e^{2\pi i\frac{1}{d_j}}$ and $\lambda_{j+1}^l=e^{2\pi i\frac{r_{j+1}^{(l)}}{d_{j+1}}}$.
Together with Theorem \ref{T2} and the claim in the proof of Proposition \ref{P5.2}, we have
\begin{equation}
\label{F5.17}1<r_{j+1}^{(l)}\mid (d_j+1).
\end{equation}
By (\ref{F5.14}), we have $\lambda_{j-1}=\lambda_j^{\frac{d_j}{d_{j-1}}}$ and
$\lambda_{j-1}=\lambda_{j+1}^{\frac{d_{j+1}}{d_{j-1}}}$. Then $\lambda_{j-1}^l=\lambda_j^{l\frac{d_j}{d_{j-1}}}=
e^{2\pi i\frac{1}{d_{j-1}}}$ and
$\lambda_{j-1}^l=\lambda_{j+1}^{l\frac{d_{j+1}}{d_{j-1}}}=e^{2\pi i\frac{r_{j+1}^{(l)}}{d_{j-1}}}$.
Combining these two equations, we have that there exists a positive integer $p_3$ such that
\begin{equation}
\label{F5.18}r_{j+1}^{(l)}=p_3d_{j-1}+1.
\end{equation}
Combining (\ref{F5.17}) and (\ref{F5.18}), we have that there exists a positive integer $K$ such that
\begin{equation}
\label{F5.21}\frac{d_j}{d_{j-1}}=(p_3d_{j-1}+1)K+p_3.
\end{equation}
Let $\Lambda_{(j-1)j(j+1)}={\rm diag}\{\lambda_{j-1},\lambda_j,\lambda_{j+1}\}$. Then
\begin{equation}
\label{F5.19}\Lambda_{(j-1)j(j+1)}^l=
{\rm diag}\{e^{2\pi i\frac{1}{d_{j-1}}},e^{2\pi i\frac{1}{d_j}},e^{2\pi i\frac{p_3d_{j-1}+1}{d_{j+1}}}\}.
\end{equation}
By Theorem \ref{T2},
there exists $H\in\mathcal{R}_{\Lambda_{(j-1)j(j+1)}}$ such that $\mathcal{N}_q(h)\geq\mathcal{N}_q(H)$
for all $q\geq1$.
Together with Theorem \ref{T1.2},
we have
$$\mathcal{N}_{d_{j-1}}(H)=\left\{\begin{matrix}2\ ,&d_{j-1}=1\\1\ ,&d_{j-1}\geq2\end{matrix}\right.$$
and
$\mathcal{N}_{d_j}(H)=1.$
Thus by (\ref{F3.7}), we have
\begin{equation}
\label{F5.16}\pi_{\tau H}(0)=d_j+d_{j-1}+1.
\end{equation}
Let $\tau H=(H_{j-1},H_j,H_{j+1})$. Since $H\in\mathcal{R}_{\Lambda_{(j-1)j(j+1)}}$, we have
$$\tau H\comp\tilde\Lambda_{(j-1)j(j+1)}^l=\tilde\Lambda_{(j-1)j(j+1)}^l(\tau H).$$
Again by (\ref{F5.19}), we have
$$H_{j-1}(x_{j-1}^{j_1},x_j,x_{j+1}^{j_2})=\sum_{i_2+j_2i_3=j_1}
c_{i_2i_3}^{(1)}x_j^{i_2}x_{j+1}^{j_2i_3}+\sum_{j_1i_1+i_2+j_2i_3=d_j+j_1}
c_{i_1i_2i_3}^{(1)}x_{j-1}^{j_1i_1}x_j^{i_2}x_{j+1}^{j_2i_3}+{\rm higher\ terms},$$

$$H_j(x_{j-1}^{j_1},x_j,x_{j+1}^{j_2})=
\sum_{j_1i_1+i_2+j_2i_3=d_j+1}c_{i_1i_2i_3}^{(2)}x_{j-1}^{j_1i_1}x_j^{i_2}x_{j+1}^{j_2i_3}+{\rm higher\ terms}$$
and
$$H_{j+1}(x_{j-1}^{j_1},x_j,x_{j+1}^{j_2})=
c_{j_2}^{(3)}x_j^{j_2}+\sum_{j_1i_1+i_2+j_2i_3=d_j+j_2}
c_{i_1i_2i_3}^{(3)}x_{j-1}^{j_1i_1}x_j^{i_2}x_{j+1}^{j_2i_3}+{\rm higher\ terms},$$
where $c_{i_2i_3}^{(1)},c_{i_1i_2i_3}^{(1)},c_{i_1i_2i_3}^{(2)},c_{j_2}^{(3)},c_{i_1i_2i_3}^{(3)}$ are complex numbers,
$j_1=\frac{d_j}{d_{j-1}}$ and $j_2=p_3d_{j-1}+1$. Then we substitute
$x_{j-1}=t_{j-1}^{p_3d_{j-1}+1}$, $x_j=t_j^{d_j+p_3d_{j-1}+2}$ and $x_{j+1}=t_{j+1}^{p_3d_{j-1}+2}$
into the above three formulas, and consider the least degree of nonzero monomials of
\begin{equation}
\label{F5.30}H_{j-1}(t_{j-1}^{p_3d_j+\frac{d_j}{d_{j-1}}},t_j^{d_j+p_3d_{j-1}+2},t_{j+1}^{(p_3d_{j-1}+2)(p_3d_{j-1}+1)}),
\end{equation}
\begin{equation}
\label{F5.31}H_j(t_{j-1}^{p_3d_j+\frac{d_j}{d_{j-1}}},t_j^{d_j+p_3d_{j-1}+2},t_{j+1}^{(p_3d_{j-1}+2)(p_3d_{j-1}+1)})
\end{equation}
and
\begin{equation}
\label{F5.32}H_{j+1}(t_{j-1}^{p_3d_j+\frac{d_j}{d_{j-1}}},t_j^{d_j+p_3d_{j-1}+2},t_{j+1}^{(p_3d_{j-1}+2)(p_3d_{j-1}+1)})
\end{equation}
respectively. Together with (\ref{F5.21}), we can obtain that the least degree of nonzero monomials of (\ref{F5.30})
is not less than the degree of
$$t_j^{(d_j+p_3d_{j-1}+2)p_3}t_{j+1}^{K(p_3d_{j-1}+1)(p_3d_{j-1}+2)},$$
and by calculation we can get that the least degree of nonzero monomials of (\ref{F5.31}) is not less than the degree of
$$t_{j+1}^{(d_j+1)(p_3d_{j-1}+2)}$$
and the least degree of nonzero monomials of (\ref{F5.32}) is not less than the degree of
$$t_j^{(d_j+p_3d_{j-1}+2)(p_3d_{j-1}+1)}.$$
Thus by Lemma \ref{P2.6} and Lemma \ref{P4}, we have $\pi_{\tau H}(0)\geq2(d_j+1)$.
This contradicts (\ref{F5.16}), and hence $j=m$.

Next we prove the necessity. Assume that $\Lambda$ is of the form (\ref{F1.1}) with $d_j\mid_{\not=}d_{j+1}$ and
$\lambda_j=\lambda_{j+1}^{\frac{d_{j+1}}{d_j}}$ for $j=1,2,\cdots,m-1$. Let $h\in\mathcal{R}_{\Lambda}$ with
$h(x_1,\cdots,x_n)=(h_1,h_2,\cdots,h_n)$ such that
\begin{align*}
h_{s_m}&=\lambda_mx_{s_m}+x_{s_0+1}^{d_1}x_{s_{m-1}+1},\\
h_{s_t}&=\lambda_tx_{s_t}+x_{s_0+1}^{d_1}x_{s_{t-1}+1}+x_{s_t+1}^{\frac{d_{t+1}}{d_t}}
\end{align*}
and
$$h_j=\lambda_{\theta(j)}x_j+x_{j+1}$$
for $t=1,2,\cdots,m-1$ and $j\in\{1,\cdots,n\}\setminus\{s_1,\cdots,s_m\}$.
By the general method in Section \ref{S2} or (\ref{FT}), we can get that $h$ satisfies (\ref{F5.13}).
Thus we complete the proof of the necessity.
\end{proof}

\subsection{Case $\uppercase\expandafter{\romannumeral4}$}
Let $\Lambda$ be an $n\times n$ matrix with the form (\ref{F1.1}) such that
$m=4$, $1<d_1\mid_{\not=} d_2$, $1<d_3\mid_{\not=} d_4$ and $(d_2,d_4)=1$.
Then, in this case, we have the following proposition.

\begin{proposition}
The matrix $\Lambda$ is not universal.\label{P5.4}
\end{proposition}
\begin{proof}
According to Theorem \ref{T2} and Proposition \ref{P5.3}, we only need to consider the case that
$\lambda_1=\lambda_2^{\frac{d_2}{d_1}}$ and $\lambda_3=\lambda_4^{\frac{d_4}{d_3}}$.
Let $\lambda_2=e^{2\pi i\frac{r_2}{d_2}}$ and $\lambda_4=e^{2\pi i\frac{r_4}{d_4}}$ such that
$(r_2,d_2)=(r_4,d_4)=1$, $1\leq r_2<d_2$ and $1\leq r_4<d_4$.
Then $\lambda_1=e^{2\pi i\frac{r_2}{d_1}}$ and $\lambda_3=e^{2\pi i\frac{r_4}{d_3}}$.
Moreover, by Lemma \ref{L1} or Chinese Remainder Theorem, there exist integers $k,p$ and $q$ such that $kr_2=pd_2+1$ and $kr_4=qd_4+1$. Thus
\begin{equation}
\label{F5.25}\lambda_j^k=e^{2\pi i\frac{1}{d_j}}
\end{equation}
for $j=1,2,3,4$.

Let $g\in\mathcal{R}_{\Lambda}$ such that $\mathcal{N}_{d_1}(g)\geq2$, $\mathcal{N}_{d_3}(g)\geq2$, $\mathcal{N}_{d_2}(g)=1$, $\mathcal{N}_{d_4}(g)=1$, $\mathcal{N}_{d_1d_3}(g)=1$
and $\mathcal{N}_{d_2d_4}(g)\geq2$.
We prove $\mathcal{N}_{d_1d_4}(g)=1$ and $\mathcal{N}_{d_2d_3}(g)=1$,
which indicates that $\Lambda$ is not universal.

We only consider the case that $k_1=k_2=k_3=k_4=1$, since we can extend this special case to general case,
as in the proof of Proposition \ref{P5.2}. Let $\tau g=(g_1,g_2,g_3,g_4)$.
Since $g\in\mathcal{R}_{\Lambda}$, we have $\tau g\comp\tilde\Lambda=\tilde\Lambda(\tau g)$.
This implies that $x_2\mid g_2$ and $x_4\mid g_4$. Let $f_2=\frac{g_2}{x_2}$ and $f_4=\frac{g_4}{x_4}$.
Then $g_1\comp\tilde\Lambda^k=\lambda_1^kg_1$, $f_2\comp\tilde\Lambda^k=g_2$, $g_3\comp\tilde\Lambda^k=\lambda_3^kg_3$
and $f_4\comp\tilde\Lambda^k=f_4$. Together with (\ref{F5.25}), we have
\begin{align*}
g_1&=c_{11}x_2^{\frac{d_2}{d_1}}+c_{12}x_1x_3^{d_3}+x_1P_{11}(x_1,x_2,x_3,
x_4)+x_2^{\frac{d_2}{d_1}}P_{12}(x_1,x_2,x_3,x_4),\\
f_2&=c_{21}x_1^{d_1}+c_{22}x_3^{d_3}+P_2(x_1,x_2,x_3,x_4),\\
g_3&=c_{31}x_4^{\frac{d_4}{d_3}}+c_{32}x_3x_1^{d_1}+x_3P_{31}(x_1,x_2,x_3,x_4)+x_4^{\frac{d_4}{d_3}}P_{32}
(x_1,x_2,x_3,x_4),\\
f_4&=c_{41}x_1^{d_1}+c_{42}x_3^{d_3}+P_4(x_1,x_2,x_3,x_4),
\end{align*}
where $c_{11},c_{12},c_{21},c_{22},c_{31},c_{32},c_{41},c_{42}$ are complex numbers,
$P_{11},P_{12},P_2,P_{31},P_{32},P_4$ are polynomials in $x_1,x_2,x_3,x_4$ such that
$$P_{11}(0,0,0,0)=0,\ P_{12}(0,0,0,0)=0,\ P_2(0,0,0,0)=0,\ P_{31}(0,0,0,0)=0,$$
$$P_{32}(0,0,0,0)=0,\ P_4(0,0,0,0)=0,\ P_{11}\comp\tilde\Lambda^k=P_{11},\ P_{12}\comp\tilde\Lambda^k=P_{12},$$
$$P_{2}\comp\tilde\Lambda^k=P_{2},\ P_{31}\comp\tilde\Lambda^k=P_{31},\ P_{32}\comp\tilde\Lambda^k=P_{32}\ {\rm and}\
P_{4}\comp\tilde\Lambda^k=P_{4}.$$
And $P_{11}$ have no monomial proportional to $x_3^{d_3}$, $P_{31}$
have no monomial proportional to $x_1^{d_1}$, and neither $P_2$ or $P_4$ have no monomial proportional to
$x_1^{d_1}$ or $x_3^{d_3}$. Furthermore, by (\ref{FT}) and Lemma \ref{P2.6}, it is not difficult to show that
\begin{enumerate}
\item $\mathcal{N}_{d_1}(g)\geq2$ implies that $P_{11}$ have no monomial proportional to $x_1^{d_1}$;
\item $\mathcal{N}_{d_3}(g)\geq2$ implies that $P_{31}$ have no monomial proportional to $x_3^{d_3}$;
\item $\mathcal{N}_{d_2}(g)=1$ implies that $c_{11}c_{21}\not=0$;
\item $\mathcal{N}_{d_4}(g)=1$ implies that $c_{31}c_{42}\not=0$;
\item Together with (1) and (2), $\mathcal{N}_{d_1d_3}(g)=1$ implies that $c_{12}c_{32}\not=0$;
\item Together with (3) and (4), $\mathcal{N}_{d_2d_4}(g)\geq2$ implies that $c_{22}c_{41}\not=0$ and
$\frac{c_{21}}{c_{22}}=\frac{c_{41}}{c_{42}}$.
\end{enumerate}
At last, based on the above results, by (\ref{FT}) and Lemma \ref{P2.6}, we have $\mathcal{N}_{d_1d_4}(g)=1$ and
$\mathcal{N}_{d_2d_3}(g)=1$.
\end{proof}

\subsection{Case $\uppercase\expandafter{\romannumeral5}$}
Let $\Lambda$ be an $n\times n$ matrix with the form (\ref{F1.1}) such that
$m=5$, $d_1=1$, $1<d_2\mid_{\not=} d_3$, $1<d_4\mid_{\not=} d_5$ and $(d_3,d_5)=1$.
Then, in this case, we have the following proposition.

\begin{proposition}
The matrix $\Lambda$ is not universal.\label{P5.5}
\end{proposition}
\begin{proof}
We will prove this proposition in a similar way as the proof of Proposition \ref{P5.4}.
First, similar to the first paragraph in the proof of Proposition \ref{P5.4},
we can get that there exists an integer $k$ such that
\begin{equation}
\label{F5.26}\lambda_{j}^k=e^{2\pi i\frac{1}{d_j}}
\end{equation}
for $j=2,3,4,5$. Second, similar to the second paragraph in the proof of Proposition \ref{P5.4},
let $g\in\mathcal{R}_{\Lambda}$ such that $\mathcal{N}_{d_2}(g)=1$, $\mathcal{N}_{d_4}(g)=1$, $\mathcal{N}_{d_3}(g)=2$,
$\mathcal{N}_{d_3d_5}(g)=1$. And we will prove $\mathcal{N}_{d_5}(g)=1$, which indicates that $\Lambda$ is not universal.
At last, similar to the third paragraph in the proof of Proposition \ref{P5.4}, we only consider the case that
$k_1=k_2=k_3=k_4=k_5=1$ and let $\tau g=(g_1,g_2,g_3,g_4,g_5)$. Since $g\in\mathcal{R}_{\Lambda}$, we have
$\tau g\comp\tilde\Lambda=\tilde\Lambda(\tau g)$. This implies $x_3\mid g_3$ and
$x_5\mid g_5$. Let $f_3=\frac{g_3}{x_3}$ and $f_5=\frac{g_5}{x_5}$.
Then $g_1\comp\tilde\Lambda^k=g_1$, $g_2\comp\tilde\Lambda^k=\lambda_2^kg_2$, $f_3\comp\tilde\Lambda^k=f_3$,
$g_4\comp\tilde\Lambda^k=\lambda_4^kg_4$ and $f_5\comp\tilde\Lambda^k=f_5$. By (\ref{F5.26}), (\ref{FT}) and
Lemma \ref{P2.6}, it is not difficult to show that
\begin{enumerate}
\item $\mathcal{N}_{d_2}(g)=1$ implies that $g_1$ has a nonzero monomial proportional to $x_2^{d_2}$;
\item $\mathcal{N}_{d_4}(g)=1$ implies that $g_1$ has a nonzero monomial proportional to $x_4^{d_4}$;
\item Together with (1), $\mathcal{N}_{d_3}(g)=2$ implies that $g_2$ has a nonzero monomial proportional to
$x_3^{\frac{d_3}{d_2}}$ and $f_3$ has no nonzero monomial proportional to $x_1$;
\item Together with (1), (2) and (3), $\mathcal{N}_{d_3d_5}(g)=1$ implies that $g_4$ has a nonzero monomial
proportional to $x_5^{\frac{d_5}{d_4}}$ and $f_5$ has a nonzero monomial proportional to $x_1$.
\end{enumerate}
Combining (2) and (4), we can get $\mathcal{N}_{d_5}(g)=1$. Thus this completes the proof of this proposition.

\end{proof}

\subsection{The proof of Theorem \ref{T1}}
\begin{proof}
We first prove the necessity. Assume that $\Lambda$ is universal. For any $d_j,d_l$, by Theorem \ref{T2} and
Proposition \ref{p5.1}, we have that $d_j,d_l$ are relatively prime or $d_j\mid d_l$ or $d_l\mid d_j$.
Then it is clear that there exist a permutation $j_1\cdots j_m$ of $12\cdots m$ and $1=l_1<l_2<\cdots<l_t\leq m$
such that
\begin{itemize}
\item $d_{j_{l_s}}\mid d_{j_{l_s+1}}\mid \cdots\mid d_{j_{l_{s+1}-1}}$ for $s=1,2,\cdots,t-1$;
\item $d_{j_{l_t}}\mid d_{j_{l_t+1}}\mid\cdots\mid d_m$;
\item $d_{j_{l_2-1}},d_{j_{l_3-1}},\cdots,d_{j_{l_t-1}},d_m$ are pairwise relatively prime.
\end{itemize}
By Lemma \ref{P2.6}, we can get that at most one of $d_1,\cdots,d_m$ is $1$. Then by Theorem \ref{T3},
Proposition \ref{P5.2} and Proposition \ref{P5.3}, we have that
$d_{j_{l_s}}\mid_{\not=}d_{j_{l_s+1}}\mid_{\not=}\cdots\mid_{\not=}d_{j_{l_{s+1}-1}}$ for
$s=1,2,\cdots,t-1$ and $d_{j_{l_t}}\mid_{\not=}d_{j_{l_t+1}}\mid_{\not=}\cdots\mid_{\not=}d_m$.
Considering $d_{j_1},d_{j_2}\cdots,d_{j_{j_t}}$ by Theorem \ref{T3} and Proposition \ref{Cr3.2},
we get that $t\leq3$ and if $t=3$,  one of $d_{j_1}$, $d_{j_2}$ and $d_{j_3}$ is $1$. At last,
according to Proposition \ref{Cr3.2}, Proposition \ref{P5.4} and Proposition \ref{P5.5},
we can get the necessity.

Next we prove the sufficiency. For all $n\times n$ invertible matrix $X$, by Lemma \ref{P4} we have $\Lambda$ and $X\Lambda X^{-1}$
have the same universal property. Thus Example \ref{EX3.1} indicates that (2) implies that $\Lambda$ is universal,
and the following example indicates that (1) implies that $\Lambda$ is universal.
\begin{example}
{\rm Let $\Lambda$ be of the form (\ref{F1.1}) with $d_j\mid_{\not=}d_{j+1}$ and
$\lambda_j=\lambda_{j+1}^{\frac{d_{j+1}}{d_j}}$ for $j=1,2,\cdots,m-1$.
Let $g\in\mathcal{O}_{\Lambda}$ with $g(x_1,\cdots,x_n)=(g_1,g_2,\cdots,g_n)$ such that
\begin{align*}
g_{s_m}&=\lambda_mx_{s_m}+x_{s_0+1}^{r_md_1}x_{s_{m-1}+1},\\
g_{s_t}&=\lambda_tx_{s_t}+x_{s_0+1}^{r_td_1}x_{s_{t-1}+1}+x_{s_t+1}^{\frac{d_{t+1}}{d_t}}
\end{align*}
and
$$g_j=\lambda_{\theta(j)}x_j+x_{j+1}$$
for $t=1,2,\cdots,m-1$ and $j\in\{1,\cdots,n\}\setminus\{s_1,\cdots,s_m\}$.

By the general method in Section \ref{S2} or (\ref{FT}), we can get that $\mathcal{N}_{d_j}(g)=r_j$ for $j=2,3,\cdots,m-1,m$
and $\mathcal{N}_{d_1}(g)=\left\{\begin{matrix}r_1\ ,&d_1>1\\r_1+1\ ,&d_1=1\end{matrix}\right.$.}
\end{example}
\end{proof}


\begin{thebibliography}{15}
\bibitem{ARN} V. I. Arnold, S. M. Gusein-Zade and A. N. Varchenko. Singularities of Differentiable Maps-Vol.
I: the Classification of Critical Points Caustics and Wave Fronts. Birkh\" auser, 1985.
\bibitem{CMY} Chow, S-N, Mallet-Paret, J. \& Yorke, J. A., A periodic orbit
index which is a bifurcation invariant. Geometric dynamics (Rio de Janeiro,
1981), Lecture Notes in Math., 1007, Springer, Berlin, 1983, 109-131.
\bibitem{Cro} J. Cronin. Analytic functional mappings. Ann. of Math. (2)
58(1953), 175-181.
\bibitem{Dol} A. Dold. Fixed point indices of iterated maps. Inventiones mathematicae 74.3(1983), 419-435.
\bibitem{FL} N. Fagella and J. Llibre. Periodic points of holomorphic maps
via Lefschetz numbers. Trans. Amer. Math. Soc. 352, No. 10(2000),
4711-4730.
\bibitem{Gor} I. Gorbovickis. On multi-dimensional Fatou bifurcation. Bulletin des Sciences Math¨¦matiques
138.3(2014), 356-375.
\bibitem{Gr1} G. Graff and J. Jezierski. On the growth of the number of
periodic points for smooth self-maps of a compact manifold. Proc. Amer.
Math. Soc. 135 , No. 10(2007), 3249-3254
\bibitem{Gr2} G. Graff and J. Jezierski. Minimal number of periodic points for $C^{1}$ self-maps of compact
simply-connected manifolds. In: Forum Mathematicum. Walter de Gruyter GmbH $\&$ Co. KG(2009), 491-509.
\bibitem{Gr4} G. Graff and J. Jezierski. Combinatorial scheme of finding minimal number of periodic points for
smooth self-maps of simply connected manifolds. Journal of Fixed Point Theory and Applications 13(1)(2013), 63-84.
\bibitem{Gr5} G. Graff and J. Jezierski. Minimal number of periodic points of smooth boundary-preserving
self-maps of simply-connected manifolds. Geometriae Dedicata, 187(1)(2017), 241-258.
\bibitem{Gr3} G. Graff and P. Nowak-Przygodzki. Fixed point indices of
iterations of $C^{1}$ maps in $\mathbb{R}^{3}$. Discrete Contin. Dyn. Syst.
16, No. 4(2006), 843-856.
\bibitem{Gr4} G. Graff and P. Nowak-Przygodzki. Periodicity of a sequence
of local fixed point indices of iterations. Osaka J. Math. 43, No. 3(2006),
485-495.
\bibitem{Mo} J. Jezierski and W. Marzantowicz. Homotopy methods in
topological fixed and periodic points theory, Series: Topological Fixed
Point Theory and Its Applications. 3. Springer, Dordrecht, 2006. xii+319 pp.
\bibitem{L} J. Llibre. Lefschetz numbers for periodic points. Contemporary
Math. 152 (1993), 215-227.
\bibitem{LL} N. G. Lloyd. Degree Theory (Cambridge Tracts in Mathematics, 73).
Cambridge University Press, Cambridge, 1978.
\bibitem{Qiao} J. Y. Qiao, H. Y. Qu and G. Y. Zhang. The numbers of periodic orbits hidden at fixed points
of holomorphic mappings. To appear in Ergodic Theory and Dynamical Systems.
\bibitem{Po1} F. Ruiz del Portal and J. M. Salazar. Fixed point indices of
the iterations of $\mathbb{R}^{3}$--homeomorphisms, preprint.
\bibitem{Po2} F. Ruiz del Portal and J. M. Salazar. Indices of the iterates
of $\mathbb{R}^{3}$--homeomorphisms at Lyapunov stable fixed points. J.
Differential Equations 244, No.5(2008), 1141-1146.
\bibitem{SS} M. Shub and D. Sullivan. A remark on the Lefschetz fixed point formula for differentiable maps.
Topology 13(1974), 189-191.
\bibitem{Zh3} G. Y. Zhang. Fixed point indices and periodic points of holomorphic mappings.
Mathematische Annalen 337.2(2007), 401-433.
\bibitem{Zh5} G. Y. Zhang. The numbers of periodic orbits hidden at fxed points of $n$-dimensional holomorphic
mappings. Ergodic Theory and Dynamical Systems, 28(2008), 1973-1989.
\bibitem{Zh2} G. Y. Zhang. The numbers of periodic orbits hidden at fixed points of $n$-dimensional holomorphic
mappings(II). Topol. Methods Nonlinear Anal.Volume 33, Number 1(2009), 65-83.
\bibitem{Zh4} G. Y. Zhang. The numbers of periodic orbits hidden at fixed points of three-dimentional holomorphic
mappings. International Journal of Mathematics, 21(05)(2010), 571-590.
\bibitem{Zh6} G. Y. Zhang. The numbers of periodic orbits hidden at fixed points of 2-dimensional holomorphic mappings.
Science China Mathematics, Volume 53, Issue 3(2010), 863-886.
\end{thebibliography}
\end{document}